\documentclass[10pt,reqno]{amsart}

\usepackage{pdfpages}
\usepackage{pdflscape} 

\usepackage{booktabs}
\usepackage{amssymb}
\usepackage{graphicx}
\usepackage{epstopdf}
\usepackage[hyphens]{url}
\usepackage[square,numbers,sort&compress]{natbib}

\usepackage{multirow}

\usepackage{tikz}
\definecolor{mediumspringgreen}{rgb}{0.0, 0.98039215, 0.60392156}

\usepackage{pgfplots} 

\usetikzlibrary{positioning}

\def\visible<#1>{}  

\usepackage[utf8]{inputenc}

\usepackage[english]{babel}
\usepackage{amsfonts}
\usepackage{amsmath}
\usepackage{latexsym}
\usepackage{color}
\usepackage{subfigure}
\usepackage{enumerate}

\usepackage{hyperref}

\usepackage{ifpdf}
\usepackage{listings}
\DeclareMathOperator    \conv           {conv}

\DeclareMathOperator    \intr                   {int}

\DeclareMathOperator    \lcm                    {lcm}

  \newcommand\bx{\bar x}\newcommand\by{\bar y}\newcommand\bz{\bar z}

\newcommand{\old}[1]{{}}

\newcommand{\bb}{\mathbb}

\newcommand{\R}{\bb R}
\newcommand{\Q}{\bb Q}
\newcommand{\Z}{\bb Z}
\newcommand{\N}{\bb N}

\newcommand\st{\mid}

\renewcommand{\P}{\mathcal{P}}

\def\st{\mid}

\newenvironment{psmallmatrixbig}{\bigl(\smallmatrix}{\endsmallmatrix\bigr)}
\newcommand\InlineFrac[2]{#1/#2}  
\newcommand\ColVec[3][\relax]
{
  \ifx#1\relax
  \bgroup\let\frac=\InlineFrac\begin{psmallmatrixbig}#2\vphantom{/}\\#3\vphantom{/}\end{psmallmatrixbig}\egroup
  \else
  \bgroup\let\frac=\InlineFrac\begin{psmallmatrixbig}\ifx#200\else#2/#1\fi\\\ifx#300\else#3/#1\fi\end{psmallmatrixbig}\egroup
  \fi
}

\newcommand\CaseM{$\wedge\kern-0.15em\wedge$}
\newcommand\CaseW{$\vee\kern-0.15em\vee$}

\newtheorem{theorem}{Theorem}[section]

\makeatletter
\newcommand\MkNewTheorem[2]{%
  \newtheorem{#1}{#2}[section]
  \expandafter\def\csname c@#1\endcsname{\c@theorem}
  \expandafter\def\csname p@#1\endcsname{\p@theorem}
  \expandafter\def\csname the#1\endcsname{\thetheorem}
  \expandafter\def\csname #1name\endcsname{#2}
}

\MkNewTheorem{corollary}{Corollary}
\MkNewTheorem{lemma}{Lemma}
\MkNewTheorem{proposition}{Proposition}
\MkNewTheorem{prop}{Proposition}
\MkNewTheorem{claim}{Claim}
\MkNewTheorem{observation}{Observation}
\MkNewTheorem{obs}{Observation}
\MkNewTheorem{conjecture}{Conjecture}
\MkNewTheorem{openquestion}{Open question}
\MkNewTheorem{question}{Question}

\theoremstyle{definition}
\MkNewTheorem{example}{Example}
\MkNewTheorem{exercise}{Exercise}
\MkNewTheorem{notation}{Notation}
\MkNewTheorem{assumption}{Assumption}
\MkNewTheorem{definition}{Definition}
\MkNewTheorem{remark}{Remark}
\MkNewTheorem{goal}{Goal}
\MkNewTheorem{problem}{Problem}

\makeatletter
\let\OurMathBbAux=\mathbb
\DeclareRobustCommand\OurMathBb{\OurMathBbAux}
\let\mathbb=\OurMathBb
\let\bfseries=\undefined
\DeclareRobustCommand\bfseries
{\not@math@alphabet\bfseries\mathbf
  \boldmath\fontseries\bfdefault\selectfont\let\OurMathBbAux=\mathbf}
\def\@thm#1#2#3{%
  \ifhmode\unskip\unskip\par\fi
  \normalfont
  \trivlist
  \let\thmheadnl\relax
  \let\thm@swap\@gobble
  \thm@notefont{\fontseries\mddefault\upshape\unboldmath}
  \thm@headpunct{.}
  \thm@headsep 5\p@ plus\p@ minus\p@\relax
  \thm@space@setup
  #1
  \@topsep \thm@preskip               
  \@topsepadd \thm@postskip           
  \def\@tempa{#2}\ifx\@empty\@tempa
    \def\@tempa{\@oparg{\@begintheorem{#3}{}}[]}%
  \else
    \refstepcounter{#2}%
    \def\@tempa{\@oparg{\@begintheorem{#3}{\csname the#2\endcsname}}[]}%
  \fi
  \@tempa
}
\makeatother

\makeatletter
\renewcommand{\pod}[1]
{\allowbreak\mathchoice{\mkern18mu}{\mkern8mu}{\mkern8mu}{\mkern8mu}(#1)}
\makeatother

\let\epsilon=\varepsilon

\let\Myunderscore=\textunderscore   
\newcommand\underscore{\Myunderscore\allowbreak}
\let\_=\underscore

\newcommand\githubsearchurl{https://github.com/mkoeppe/cutgeneratingfunctionology/search}
\pgfkeyssetvalue{/sagefunc/bccz_counterexample}{\href{\githubsearchurl?q=\%22def+bccz_counterexample(\%22}{\sage{bccz\underscore{}counterexample}}}
\pgfkeyssetvalue{/sagefunc/bcdsp_arbitrary_slope}{\href{\githubsearchurl?q=\%22def+bcdsp_arbitrary_slope(\%22}{\sage{bcdsp\underscore{}arbitrary\underscore{}slope}}}
\pgfkeyssetvalue{/sagefunc/bhk_gmi_irrational}{\href{\githubsearchurl?q=\%22def+bhk_gmi_irrational(\%22}{\sage{bhk\underscore{}gmi\underscore{}irrational}}}
\pgfkeyssetvalue{/sagefunc/bhk_irrational}{\href{\githubsearchurl?q=\%22def+bhk_irrational(\%22}{\sage{bhk\underscore{}irrational}}}
\pgfkeyssetvalue{/sagefunc/bhk_slant_irrational}{\href{\githubsearchurl?q=\%22def+bhk_slant_irrational(\%22}{\sage{bhk\underscore{}slant\underscore{}irrational}}}
\pgfkeyssetvalue{/sagefunc/chen_4_slope}{\href{\githubsearchurl?q=\%22def+chen_4_slope(\%22}{\sage{chen\underscore{}4\underscore{}slope}}}
\pgfkeyssetvalue{/sagefunc/dg_2_step_mir}{\href{\githubsearchurl?q=\%22def+dg_2_step_mir(\%22}{\sage{dg\underscore{}2\underscore{}step\underscore{}mir}}}
\pgfkeyssetvalue{/sagefunc/dg_2_step_mir_limit}{\href{\githubsearchurl?q=\%22def+dg_2_step_mir_limit(\%22}{\sage{dg\underscore{}2\underscore{}step\underscore{}mir\underscore{}limit}}}
\pgfkeyssetvalue{/sagefunc/drlm_2_slope_limit}{\href{\githubsearchurl?q=\%22def+drlm_2_slope_limit(\%22}{\sage{drlm\underscore{}2\underscore{}slope\underscore{}limit}}}
\pgfkeyssetvalue{/sagefunc/drlm_3_slope_limit}{\href{\githubsearchurl?q=\%22def+drlm_3_slope_limit(\%22}{\sage{drlm\underscore{}3\underscore{}slope\underscore{}limit}}}
\pgfkeyssetvalue{/sagefunc/drlm_backward_3_slope}{\href{\githubsearchurl?q=\%22def+drlm_backward_3_slope(\%22}{\sage{drlm\underscore{}backward\underscore{}3\underscore{}slope}}}
\pgfkeyssetvalue{/sagefunc/gj_2_slope}{\href{\githubsearchurl?q=\%22def+gj_2_slope(\%22}{\sage{gj\underscore{}2\underscore{}slope}}}
\pgfkeyssetvalue{/sagefunc/gj_2_slope_repeat}{\href{\githubsearchurl?q=\%22def+gj_2_slope_repeat(\%22}{\sage{gj\underscore{}2\underscore{}slope\underscore{}repeat}}}
\pgfkeyssetvalue{/sagefunc/gj_forward_3_slope}{\href{\githubsearchurl?q=\%22def+gj_forward_3_slope(\%22}{\sage{gj\underscore{}forward\underscore{}3\underscore{}slope}}}
\pgfkeyssetvalue{/sagefunc/gmic}{\href{\githubsearchurl?q=\%22def+gmic(\%22}{\sage{gmic}}}
\pgfkeyssetvalue{/sagefunc/hildebrand_2_sided_discont_1_slope_1}{\href{\githubsearchurl?q=\%22def+hildebrand_2_sided_discont_1_slope_1(\%22}{\sage{hildebrand\underscore{}2\underscore{}sided\underscore{}discont\underscore{}1\underscore{}slope\underscore{}1}}}
\pgfkeyssetvalue{/sagefunc/hildebrand_2_sided_discont_2_slope_1}{\href{\githubsearchurl?q=\%22def+hildebrand_2_sided_discont_2_slope_1(\%22}{\sage{hildebrand\underscore{}2\underscore{}sided\underscore{}discont\underscore{}2\underscore{}slope\underscore{}1}}}
\pgfkeyssetvalue{/sagefunc/hildebrand_5_slope_22_1}{\href{\githubsearchurl?q=\%22def+hildebrand_5_slope_22_1(\%22}{\sage{hildebrand\underscore{}5\underscore{}slope\underscore{}22\underscore{}1}}}
\pgfkeyssetvalue{/sagefunc/hildebrand_5_slope_24_1}{\href{\githubsearchurl?q=\%22def+hildebrand_5_slope_24_1(\%22}{\sage{hildebrand\underscore{}5\underscore{}slope\underscore{}24\underscore{}1}}}
\pgfkeyssetvalue{/sagefunc/hildebrand_5_slope_28_1}{\href{\githubsearchurl?q=\%22def+hildebrand_5_slope_28_1(\%22}{\sage{hildebrand\underscore{}5\underscore{}slope\underscore{}28\underscore{}1}}}
\pgfkeyssetvalue{/sagefunc/hildebrand_discont_3_slope_1}{\href{\githubsearchurl?q=\%22def+hildebrand_discont_3_slope_1(\%22}{\sage{hildebrand\underscore{}discont\underscore{}3\underscore{}slope\underscore{}1}}}
\pgfkeyssetvalue{/sagefunc/kf_n_step_mir}{\href{\githubsearchurl?q=\%22def+kf_n_step_mir(\%22}{\sage{kf\underscore{}n\underscore{}step\underscore{}mir}}}
\pgfkeyssetvalue{/sagefunc/kzh_10_slope_1}{\href{\githubsearchurl?q=\%22def+kzh_10_slope_1(\%22}{\sage{kzh\underscore{}10\underscore{}slope\underscore{}1}}}
\pgfkeyssetvalue{/sagefunc/kzh_28_slope_1}{\href{\githubsearchurl?q=\%22def+kzh_28_slope_1(\%22}{\sage{kzh\underscore{}28\underscore{}slope\underscore{}1}}}
\pgfkeyssetvalue{/sagefunc/kzh_28_slope_2}{\href{\githubsearchurl?q=\%22def+kzh_28_slope_2(\%22}{\sage{kzh\underscore{}28\underscore{}slope\underscore{}2}}}
\pgfkeyssetvalue{/sagefunc/kzh_3_slope_param_extreme_1}{\href{\githubsearchurl?q=\%22def+kzh_3_slope_param_extreme_1(\%22}{\sage{kzh\underscore{}3\underscore{}slope\underscore{}param\underscore{}extreme\underscore{}1}}}
\pgfkeyssetvalue{/sagefunc/kzh_3_slope_param_extreme_2}{\href{\githubsearchurl?q=\%22def+kzh_3_slope_param_extreme_2(\%22}{\sage{kzh\underscore{}3\underscore{}slope\underscore{}param\underscore{}extreme\underscore{}2}}}
\pgfkeyssetvalue{/sagefunc/kzh_4_slope_param_extreme_1}{\href{\githubsearchurl?q=\%22def+kzh_4_slope_param_extreme_1(\%22}{\sage{kzh\underscore{}4\underscore{}slope\underscore{}param\underscore{}extreme\underscore{}1}}}
\pgfkeyssetvalue{/sagefunc/kzh_5_slope_fulldim_1}{\href{\githubsearchurl?q=\%22def+kzh_5_slope_fulldim_1(\%22}{\sage{kzh\underscore{}5\underscore{}slope\underscore{}fulldim\underscore{}1}}}
\pgfkeyssetvalue{/sagefunc/kzh_5_slope_fulldim_2}{\href{\githubsearchurl?q=\%22def+kzh_5_slope_fulldim_2(\%22}{\sage{kzh\underscore{}5\underscore{}slope\underscore{}fulldim\underscore{}2}}}
\pgfkeyssetvalue{/sagefunc/kzh_5_slope_fulldim_3}{\href{\githubsearchurl?q=\%22def+kzh_5_slope_fulldim_3(\%22}{\sage{kzh\underscore{}5\underscore{}slope\underscore{}fulldim\underscore{}3}}}
\pgfkeyssetvalue{/sagefunc/kzh_5_slope_fulldim_4}{\href{\githubsearchurl?q=\%22def+kzh_5_slope_fulldim_4(\%22}{\sage{kzh\underscore{}5\underscore{}slope\underscore{}fulldim\underscore{}4}}}
\pgfkeyssetvalue{/sagefunc/kzh_5_slope_fulldim_5}{\href{\githubsearchurl?q=\%22def+kzh_5_slope_fulldim_5(\%22}{\sage{kzh\underscore{}5\underscore{}slope\underscore{}fulldim\underscore{}5}}}
\pgfkeyssetvalue{/sagefunc/kzh_5_slope_fulldim_covers_1}{\href{\githubsearchurl?q=\%22def+kzh_5_slope_fulldim_covers_1(\%22}{\sage{kzh\underscore{}5\underscore{}slope\underscore{}fulldim\underscore{}covers\underscore{}1}}}
\pgfkeyssetvalue{/sagefunc/kzh_5_slope_fulldim_covers_2}{\href{\githubsearchurl?q=\%22def+kzh_5_slope_fulldim_covers_2(\%22}{\sage{kzh\underscore{}5\underscore{}slope\underscore{}fulldim\underscore{}covers\underscore{}2}}}
\pgfkeyssetvalue{/sagefunc/kzh_5_slope_fulldim_covers_3}{\href{\githubsearchurl?q=\%22def+kzh_5_slope_fulldim_covers_3(\%22}{\sage{kzh\underscore{}5\underscore{}slope\underscore{}fulldim\underscore{}covers\underscore{}3}}}
\pgfkeyssetvalue{/sagefunc/kzh_5_slope_fulldim_covers_4}{\href{\githubsearchurl?q=\%22def+kzh_5_slope_fulldim_covers_4(\%22}{\sage{kzh\underscore{}5\underscore{}slope\underscore{}fulldim\underscore{}covers\underscore{}4}}}
\pgfkeyssetvalue{/sagefunc/kzh_5_slope_fulldim_covers_5}{\href{\githubsearchurl?q=\%22def+kzh_5_slope_fulldim_covers_5(\%22}{\sage{kzh\underscore{}5\underscore{}slope\underscore{}fulldim\underscore{}covers\underscore{}5}}}
\pgfkeyssetvalue{/sagefunc/kzh_5_slope_fulldim_covers_6}{\href{\githubsearchurl?q=\%22def+kzh_5_slope_fulldim_covers_6(\%22}{\sage{kzh\underscore{}5\underscore{}slope\underscore{}fulldim\underscore{}covers\underscore{}6}}}
\pgfkeyssetvalue{/sagefunc/kzh_5_slope_q22_f10_1}{\href{\githubsearchurl?q=\%22def+kzh_5_slope_q22_f10_1(\%22}{\sage{kzh\underscore{}5\underscore{}slope\underscore{}q22\underscore{}f10\underscore{}1}}}
\pgfkeyssetvalue{/sagefunc/kzh_5_slope_q22_f10_2}{\href{\githubsearchurl?q=\%22def+kzh_5_slope_q22_f10_2(\%22}{\sage{kzh\underscore{}5\underscore{}slope\underscore{}q22\underscore{}f10\underscore{}2}}}
\pgfkeyssetvalue{/sagefunc/kzh_5_slope_q22_f10_3}{\href{\githubsearchurl?q=\%22def+kzh_5_slope_q22_f10_3(\%22}{\sage{kzh\underscore{}5\underscore{}slope\underscore{}q22\underscore{}f10\underscore{}3}}}
\pgfkeyssetvalue{/sagefunc/kzh_5_slope_q22_f10_4}{\href{\githubsearchurl?q=\%22def+kzh_5_slope_q22_f10_4(\%22}{\sage{kzh\underscore{}5\underscore{}slope\underscore{}q22\underscore{}f10\underscore{}4}}}
\pgfkeyssetvalue{/sagefunc/kzh_5_slope_q22_f2_1}{\href{\githubsearchurl?q=\%22def+kzh_5_slope_q22_f2_1(\%22}{\sage{kzh\underscore{}5\underscore{}slope\underscore{}q22\underscore{}f2\underscore{}1}}}
\pgfkeyssetvalue{/sagefunc/kzh_6_slope_1}{\href{\githubsearchurl?q=\%22def+kzh_6_slope_1(\%22}{\sage{kzh\underscore{}6\underscore{}slope\underscore{}1}}}
\pgfkeyssetvalue{/sagefunc/kzh_6_slope_fulldim_covers_1}{\href{\githubsearchurl?q=\%22def+kzh_6_slope_fulldim_covers_1(\%22}{\sage{kzh\underscore{}6\underscore{}slope\underscore{}fulldim\underscore{}covers\underscore{}1}}}
\pgfkeyssetvalue{/sagefunc/kzh_6_slope_fulldim_covers_2}{\href{\githubsearchurl?q=\%22def+kzh_6_slope_fulldim_covers_2(\%22}{\sage{kzh\underscore{}6\underscore{}slope\underscore{}fulldim\underscore{}covers\underscore{}2}}}
\pgfkeyssetvalue{/sagefunc/kzh_6_slope_fulldim_covers_3}{\href{\githubsearchurl?q=\%22def+kzh_6_slope_fulldim_covers_3(\%22}{\sage{kzh\underscore{}6\underscore{}slope\underscore{}fulldim\underscore{}covers\underscore{}3}}}
\pgfkeyssetvalue{/sagefunc/kzh_6_slope_fulldim_covers_4}{\href{\githubsearchurl?q=\%22def+kzh_6_slope_fulldim_covers_4(\%22}{\sage{kzh\underscore{}6\underscore{}slope\underscore{}fulldim\underscore{}covers\underscore{}4}}}
\pgfkeyssetvalue{/sagefunc/kzh_6_slope_fulldim_covers_5}{\href{\githubsearchurl?q=\%22def+kzh_6_slope_fulldim_covers_5(\%22}{\sage{kzh\underscore{}6\underscore{}slope\underscore{}fulldim\underscore{}covers\underscore{}5}}}
\pgfkeyssetvalue{/sagefunc/kzh_7_slope_1}{\href{\githubsearchurl?q=\%22def+kzh_7_slope_1(\%22}{\sage{kzh\underscore{}7\underscore{}slope\underscore{}1}}}
\pgfkeyssetvalue{/sagefunc/kzh_7_slope_2}{\href{\githubsearchurl?q=\%22def+kzh_7_slope_2(\%22}{\sage{kzh\underscore{}7\underscore{}slope\underscore{}2}}}
\pgfkeyssetvalue{/sagefunc/kzh_7_slope_3}{\href{\githubsearchurl?q=\%22def+kzh_7_slope_3(\%22}{\sage{kzh\underscore{}7\underscore{}slope\underscore{}3}}}
\pgfkeyssetvalue{/sagefunc/kzh_7_slope_4}{\href{\githubsearchurl?q=\%22def+kzh_7_slope_4(\%22}{\sage{kzh\underscore{}7\underscore{}slope\underscore{}4}}}
\pgfkeyssetvalue{/sagefunc/ll_strong_fractional}{\href{\githubsearchurl?q=\%22def+ll_strong_fractional(\%22}{\sage{ll\underscore{}strong\underscore{}fractional}}}
\pgfkeyssetvalue{/sagefunc/mlr_cpl3_a_2_slope}{\href{\githubsearchurl?q=\%22def+mlr_cpl3_a_2_slope(\%22}{\sage{mlr\underscore{}cpl3\underscore{}a\underscore{}2\underscore{}slope}}}
\pgfkeyssetvalue{/sagefunc/mlr_cpl3_b_3_slope}{\href{\githubsearchurl?q=\%22def+mlr_cpl3_b_3_slope(\%22}{\sage{mlr\underscore{}cpl3\underscore{}b\underscore{}3\underscore{}slope}}}
\pgfkeyssetvalue{/sagefunc/mlr_cpl3_c_3_slope}{\href{\githubsearchurl?q=\%22def+mlr_cpl3_c_3_slope(\%22}{\sage{mlr\underscore{}cpl3\underscore{}c\underscore{}3\underscore{}slope}}}
\pgfkeyssetvalue{/sagefunc/mlr_cpl3_d_3_slope}{\href{\githubsearchurl?q=\%22def+mlr_cpl3_d_3_slope(\%22}{\sage{mlr\underscore{}cpl3\underscore{}d\underscore{}3\underscore{}slope}}}
\pgfkeyssetvalue{/sagefunc/mlr_cpl3_f_2_or_3_slope}{\href{\githubsearchurl?q=\%22def+mlr_cpl3_f_2_or_3_slope(\%22}{\sage{mlr\underscore{}cpl3\underscore{}f\underscore{}2\underscore{}or\underscore{}3\underscore{}slope}}}
\pgfkeyssetvalue{/sagefunc/mlr_cpl3_g_3_slope}{\href{\githubsearchurl?q=\%22def+mlr_cpl3_g_3_slope(\%22}{\sage{mlr\underscore{}cpl3\underscore{}g\underscore{}3\underscore{}slope}}}
\pgfkeyssetvalue{/sagefunc/mlr_cpl3_h_2_slope}{\href{\githubsearchurl?q=\%22def+mlr_cpl3_h_2_slope(\%22}{\sage{mlr\underscore{}cpl3\underscore{}h\underscore{}2\underscore{}slope}}}
\pgfkeyssetvalue{/sagefunc/mlr_cpl3_k_2_slope}{\href{\githubsearchurl?q=\%22def+mlr_cpl3_k_2_slope(\%22}{\sage{mlr\underscore{}cpl3\underscore{}k\underscore{}2\underscore{}slope}}}
\pgfkeyssetvalue{/sagefunc/mlr_cpl3_l_2_slope}{\href{\githubsearchurl?q=\%22def+mlr_cpl3_l_2_slope(\%22}{\sage{mlr\underscore{}cpl3\underscore{}l\underscore{}2\underscore{}slope}}}
\pgfkeyssetvalue{/sagefunc/mlr_cpl3_n_3_slope}{\href{\githubsearchurl?q=\%22def+mlr_cpl3_n_3_slope(\%22}{\sage{mlr\underscore{}cpl3\underscore{}n\underscore{}3\underscore{}slope}}}
\pgfkeyssetvalue{/sagefunc/mlr_cpl3_o_2_slope}{\href{\githubsearchurl?q=\%22def+mlr_cpl3_o_2_slope(\%22}{\sage{mlr\underscore{}cpl3\underscore{}o\underscore{}2\underscore{}slope}}}
\pgfkeyssetvalue{/sagefunc/mlr_cpl3_p_2_slope}{\href{\githubsearchurl?q=\%22def+mlr_cpl3_p_2_slope(\%22}{\sage{mlr\underscore{}cpl3\underscore{}p\underscore{}2\underscore{}slope}}}
\pgfkeyssetvalue{/sagefunc/mlr_cpl3_q_2_slope}{\href{\githubsearchurl?q=\%22def+mlr_cpl3_q_2_slope(\%22}{\sage{mlr\underscore{}cpl3\underscore{}q\underscore{}2\underscore{}slope}}}
\pgfkeyssetvalue{/sagefunc/mlr_cpl3_r_2_slope}{\href{\githubsearchurl?q=\%22def+mlr_cpl3_r_2_slope(\%22}{\sage{mlr\underscore{}cpl3\underscore{}r\underscore{}2\underscore{}slope}}}
\pgfkeyssetvalue{/sagefunc/rlm_dpl1_extreme_3a}{\href{\githubsearchurl?q=\%22def+rlm_dpl1_extreme_3a(\%22}{\sage{rlm\underscore{}dpl1\underscore{}extreme\underscore{}3a}}}
\pgfkeyssetvalue{/sagefunc/automorphism}{\href{\githubsearchurl?q=\%22def+automorphism(\%22}{\sage{automorphism}}}
\pgfkeyssetvalue{/sagefunc/multiplicative_homomorphism}{\href{\githubsearchurl?q=\%22def+multiplicative_homomorphism(\%22}{\sage{multiplicative\underscore{}homomorphism}}}
\pgfkeyssetvalue{/sagefunc/projected_sequential_merge}{\href{\githubsearchurl?q=\%22def+projected_sequential_merge(\%22}{\sage{projected\underscore{}sequential\underscore{}merge}}}
\pgfkeyssetvalue{/sagefunc/restrict_to_finite_group}{\href{\githubsearchurl?q=\%22def+restrict_to_finite_group(\%22}{\sage{restrict\underscore{}to\underscore{}finite\underscore{}group}}}
\pgfkeyssetvalue{/sagefunc/interpolate_to_infinite_group}{\href{\githubsearchurl?q=\%22def+interpolate_to_infinite_group(\%22}{\sage{interpolate\underscore{}to\underscore{}infinite\underscore{}group}}}
\pgfkeyssetvalue{/sagefunc/two_slope_fill_in}{\href{\githubsearchurl?q=\%22def+two_slope_fill_in(\%22}{\sage{two\underscore{}slope\underscore{}fill\underscore{}in}}}
\pgfkeyssetvalue{/sagefunc/generate_example_e_for_psi_n}{\href{\githubsearchurl?q=\%22def+generate_example_e_for_psi_n(\%22}{\sage{generate\underscore{}example\underscore{}e\underscore{}for\underscore{}psi\underscore{}n}}}
\pgfkeyssetvalue{/sagefunc/chen_3_slope_not_extreme}{\href{\githubsearchurl?q=\%22def+chen_3_slope_not_extreme(\%22}{\sage{chen\underscore{}3\underscore{}slope\underscore{}not\underscore{}extreme}}}
\pgfkeyssetvalue{/sagefunc/psi_n_in_bccz_counterexample_construction}{\href{\githubsearchurl?q=\%22def+psi_n_in_bccz_counterexample_construction(\%22}{\sage{psi\underscore{}n\underscore{}in\underscore{}bccz\underscore{}counterexample\underscore{}construction}}}
\pgfkeyssetvalue{/sagefunc/gomory_fractional}{\href{\githubsearchurl?q=\%22def+gomory_fractional(\%22}{\sage{gomory\underscore{}fractional}}}
\pgfkeyssetvalue{/sagefunc/not_minimal_2}{\href{\githubsearchurl?q=\%22def+not_minimal_2(\%22}{\sage{not\underscore{}minimal\underscore{}2}}}
\pgfkeyssetvalue{/sagefunc/not_extreme_1}{\href{\githubsearchurl?q=\%22def+not_extreme_1(\%22}{\sage{not\underscore{}extreme\underscore{}1}}}
\pgfkeyssetvalue{/sagefunc/kzh_2q_example_1}{\href{\githubsearchurl?q=\%22def+kzh_2q_example_1(\%22}{\sage{kzh\underscore{}2q\underscore{}example\underscore{}1}}}
\pgfkeyssetvalue{/sagefunc/zhou_two_sided_discontinuous_cannot_assume_any_continuity}{\href{\githubsearchurl?q=\%22def+zhou_two_sided_discontinuous_cannot_assume_any_continuity(\%22}{\sage{zhou\underscore{}two\underscore{}sided\underscore{}discontinuous\underscore{}cannot\underscore{}assume\underscore{}any\underscore{}continuity}}}
\pgfkeyssetvalue{/sagefunc/extremality_test}{\href{\githubsearchurl?q=\%22def+extremality_test(\%22}{\sage{extremality\underscore{}test}}}
\pgfkeyssetvalue{/sagefunc/plot_2d_diagram}{\href{\githubsearchurl?q=\%22def+plot_2d_diagram(\%22}{\sage{plot\underscore{}2d\underscore{}diagram}}}
\pgfkeyssetvalue{/sagefunc/nice_field_values}{\href{\githubsearchurl?q=\%22def+nice_field_values(\%22}{\sage{nice\underscore{}field\underscore{}values}}}
\pgfkeyssetvalue{/sagefunc/ParametricRealFieldElement}{\href{\githubsearchurl?q=\%22def+ParametricRealFieldElement(\%22}{\sage{ParametricRealFieldElement}}}
\pgfkeyssetvalue{/sagefunc/ParametricRealField}{\href{\githubsearchurl?q=\%22def+ParametricRealField(\%22}{\sage{ParametricRealField}}}
\pgfkeyssetvalue{/sagefunc/kzh_minimal_has_only_crazy_perturbation_1}{\href{\githubsearchurl?q=\%22def+kzh_minimal_has_only_crazy_perturbation_1(\%22}{\sage{kzh\underscore{}minimal\underscore{}has\underscore{}only\underscore{}crazy\underscore{}perturbation\underscore{}1}}}
\pgfkeyssetvalue{/sagefunc/bcds_discontinuous_everywhere}{\href{\githubsearchurl?q=\%22def+bcds_discontinuous_everywhere(\%22}{\sage{bcds\underscore{}discontinuous\underscore{}everywhere}}}

\DeclareRobustCommand\sage[1]{\textsf{\upshape #1}}
\DeclareRobustCommand\sagefunc[1]{\pgfkeys{/sagefunc/#1}}
\DeclareRobustCommand\sagefuncgraph[1]{\raisebox{-0.08ex}{\includegraphics[height=2ex,width=2.5em]{funcgraphs/#1}}}
\DeclareRobustCommand\sagefuncwithgraph[1]{\sagefunc{#1} \sagefuncgraph{#1}}

\newcommand\DIFFPROTECT[1]{#1}

\title{All Cyclic Group Facets Inject}

\dedicatory{Dedicated to Professor Ellis L. Johnson on the occasion of his
  eightieth birthday} 

\thanks{The authors gratefully acknowledge partial support from the National Science
  Foundation through grant DMS-1320051, awarded to M.~K\"oppe.}

\author{Matthias K\"oppe}
\address{Matthias K\"oppe: Dept.\ of Mathematics, University of California, Davis}
\email{mkoeppe@math.ucdavis.edu}

\author{Yuan Zhou}
\address{Yuan Zhou: Dept.\ of Mathematics, University of Kentucky}
\email{yuan.zhou@uky.edu}

\date{$\relax$Revision: 3220 $ - \ $Date: 2019-04-16 21:10:37 -0700 (Tue, 16 Apr 2019) $ $\!\!\!}

\usepackage[normalem]{ulem}

\newcommand\Figure[2][\relax]{%
  \begin{figure}[h!]
    \includegraphics[width=.8\textwidth]{#2}
    \caption{\ifx#1\relax#2\else#1\fi}
  \end{figure}
}

\begin{document}
 \newcommand{\tgreen}[1]{\textsf{\textcolor {ForestGreen} {#1}}}
 \newcommand{\tred}[1]{\texttt{\textcolor {red} {#1}}}
 \newcommand{\tblue}[1]{\textcolor {blue} {#1}}

\begin{abstract}
  We give a variant of Basu--Hildebrand--Molinaro's approximation theorem for
  continuous minimal valid functions for Gomory--Johnson's infinite group
  problem by piecewise linear two-slope extreme functions [\emph{Minimal
    cut-generating functions are nearly extreme}, IPCO 2016].  Our theorem is
  for piecewise linear minimal valid functions that have only rational
  breakpoints (in $1/q\,\mathbb{Z}$ for some $q\in \mathbb{N}$) and that take
  rational values at the breakpoints.  In contrast to Basu et al.'s
  construction, our construction preserves all function values on
  $1/q\,\mathbb{Z}$.  As a corollary, we obtain that every extreme function
  for the finite group problem on $1/q\,\mathbb{Z}$ is the restriction of a
  continuous piecewise linear two-slope extreme function for the infinite
  group problem with breakpoints on a refinement $1/(Mq)\,\mathbb{Z}$ for some
  $M\in \mathbb{N}$.  In combination with Gomory's master theorem [\emph{Some
  Polyhedra related to Combinatorial Problems}, Lin.\@ Alg.\@ Appl.\@ 2 (1969),
  451--558], this shows that the infinite group problem is the correct master
  problem for facets (extreme functions) of 1-row group relaxations.
\end{abstract}
\maketitle


\section{Introduction}

Gomory introduced the finite group relaxations of integer programming problems
in his seminal paper~\cite{gom}, expanding upon \cite{MR0182454}.
Let $G$ be an abelian group, written additively (as a $\Z$-module).  Finite
groups $G$ arise concretely as $B^{-1}\Z^k/\Z^k$, 
where $B$ is a basis matrix in the simplex method, applied to the continuous
relaxation of a pure integer program with all-integer data.  For concreteness
and simplicity, throughout this paper we consider all groups $G$ as subgroups
of the infinite group~$\Q^k/\Z^k \subset \R^k/\Z^k$ for some~$k$.  
For example, the cyclic group $C_q$ is
realized as $\frac1q\Z/\Z \subset \Q/\Z \subset \R/\Z$.  

Let $P \subseteq G$ be a finite subset and let $f$ be an element of
$G\setminus\{0\}$.  Consider the set of functions $y \colon P \to \Z_+$
satisfying the constraint (``group equation''):
\begin{equation}
  \begin{aligned}
    &\sum_{p \in P} y(p) \, p = f 
    .\label{eq:group-eqn}
  \end{aligned}
\end{equation}
(The summation takes place in $G$ and hence the equation is ``modulo 1.'')
Denote by
$R_f(P) 
$ the convex hull of all solutions $y \in \Z_+^P$
to~\eqref{eq:group-eqn}.  
The set $R_f(P)$, if nonempty, is a polyhedron in $\R_+^P$ of \emph{blocking
  type}, i.e., its recession cone is $\R_+^P$; see, for example, \cite[section
6.1]{ccz-ipbook}.  It is known as the \emph{corner polyhedron}.  If $G$ is a
finite group and $P = G$, then one speaks of a \emph{(finite) master group relaxation}
and a \emph{master corner polyhedron}; we will comment on the meaning of the
word ``master'' in these notions shortly.

The question arose how to make effective use of the group relaxation in
solvers.  An early emphasis lay on the primal aspects of the problem, such as
the use of dynamic programming to generate ``paths'' (solutions); see the
survey \cite[section 19.3]{Richard-Dey-2010:50-year-survey} and also Gomory's
essay \cite{gomory2007atoms}.  
Trivially, a solution $y\in\Z_+^P$ of $R_f(P)$ injects 
into larger problems, in particular master problems $R_f(G)$, by setting
\begin{equation}
  \label{eq:primal-injection}
  y(p) = 0 \quad\text{for} \quad p \notin P.
\end{equation}

The renewed interest in the group approach in recent years, however, has
almost exclusively focused on the dual aspects, i.e., on generating valid
inequalities, which can be applied directly to the original integer program.
This is also the viewpoint of the present paper.  For an overview we refer to
the surveys \cite[sections 19.4--19.6]{Richard-Dey-2010:50-year-survey} and
\cite{igp_survey,igp_survey_part_2}.  For recent developments we refer to
\cite{DiSumma-2018:piecewise-smooth-piecewise-linear,basu2017optimal,hildebrand-koeppe-zhou:algo-paper-abstract-ipco}. 

\subsection{Valid functions}

Because $R_f(P)$ is a polyhedron of blocking type, all non-trivial valid
inequalities can be normalized to the form $\sum_{p \in P} \pi(p) y(p) \geq 1$
for some non-negative function $\pi\colon P\to \R_+$.  Such functions~$\pi$
are called \emph{valid functions}.  A valid function $\pi$ is said to be
\emph{minimal} if there is no other valid function $\pi^\dagger \neq \pi$ such that
$\pi^\dagger \leq \pi$ pointwise.  The set of minimal valid functions for arbitrary problems
$R_f(P)$ can have a complicated structure; but for \emph{master} problems
$R_f(G)$, Gomory \cite{gom} gave the following important characterization: 
\begin{subequations}\label{eq:minimal}
  \begin{align}
    &\pi(x) \geq 0 &&\text{for } x \in G, \label{eq:minimal:nonneg}\\
    &\pi(0)=0, \ \pi(f)=1, \label{eq:minimal:01}\\
    &\Delta\pi(x,y) \geq 0  
    &&  \text{for } x,y \in G
                   && \text{(subadditivity),}\label{eq:minimal:subadd}\\
    &\Delta\pi(x, f-x) = 0 
    && \text{for } x\in G 
                   && \text{(symmetry condition),} \label{eq:minimal:symm}
  \end{align}
\end{subequations}
where $\Delta\pi(x,y) = \pi(x)+\pi(y) -\pi(x+y)$ is the \emph{subadditivity slack} function.
The functions~$\pi$ giving rise to (non-trivial) facet-defining inequalities
are referred to as \emph{facets}.  They are the functions that are
\emph{extreme} among the minimal valid functions, i.e., they satisfy
\begin{equation}
   \begin{aligned}
     \text{if $\pi^+$ and $\pi^-$ are minimal and $\pi = \tfrac12 (\pi^++\pi^-)$}\\
     \text{then $\pi = \pi^+ = \pi^-$}. 
     \label{eq:facet-definition-like-extreme-function}
   \end{aligned}
\end{equation}
In the present paper, we take the classic viewpoint of considering the
facets (extreme functions) the ``important'' valid functions.


\subsection{Injections of valid functions into larger finite group problems by
homomorphism}

The group relaxations can be seen as models that capture a nontrivial amount
of the strength of integer programs, 
but which are sufficiently structured to allow us to apply some analysis.  For
example, the master problems admit additional symmetries, not found in the
original integer programs, in the form of group automorphisms, which apply to
the solutions and also to the valid inequalities.  Gomory \cite{gom} exploited
this fact by enumerating facets of the master corner polyhedra up to
automorphisms.   

In addition, via pullbacks by
group homomorphisms, the set of valid inequalities for a master problem
for a group~$H$ injects into the set of valid inequalities for master problems
for groups~$G$, where $H$ is a homomorphic image of~$G$.  This hierarchy of injections preserves
facetness.
\def\Homo{A}
\begin{theorem}[{homomorphism theorem \cite[Theorem 19]{gom}}]
  \label{thm:homomorphism}
  \begin{enumerate}[\rm(a)]
  \item
    Let $\Homo\colon G \to H$ be a homomorphism onto $H$ with kernel $K$ and
    $f \notin K$.  Let $\check\pi\colon H\to\R$ be a facet 
    of $R_{\Homo f}(H)$.  Define the pullback $\pi\colon G\to \R$ of $\check\pi$ by
    $\Homo$ as $\pi(g) = \check\pi(\Homo g)$.
    Then 
    $\pi$ is a facet of $R_{f}(G)$.
  \item In the other direction, if a function $\pi\colon G\to \R$ defines a
    facet of~$R_f(G)$ and is constant on the cosets of~$K$, then~$\pi(g)$
    factors through the canonical homomorphism $\Homo\colon G\to G/K$,
    $\pi(g) = \check\pi(\Homo g)$, and $\check\pi$ defines a facet of~$R_f(G/K)$.
  \end{enumerate}
\end{theorem}

However,
we do not know a natural ``universal'' master problem into which all facets
inject via pullbacks of homomorphisms.  

\subsection{Gomory--Johnson's infinite group problem}
In particular, consider the
\emph{infinite group problem} introduced by Gomory and Johnson in their
remarkable papers~\cite{infinite,infinite2}, which will play an import r\^ole later in the
present paper.  Here one takes $P = G = \R^k/\Z^k$ 
and
defines $R_f(G)$ to be the convex hull of the finite-support functions
$y\colon G\to \Z_+$ satisfying~\eqref{eq:group-eqn}.  Minimal functions form a
convex set in the space of functions $\pi\colon G\to\R$, again characterized
by \eqref{eq:minimal}.  For the infinite group problem, there is a subtle
difference between the notions of extreme functions, defined
by~\eqref{eq:facet-definition-like-extreme-function}, and facets; see
\cite{koeppe-zhou:discontinuous-facets}. 
However, for the important case of continuous
piecewise linear functions of $\R/\Z$, both notions agree (see \cite[Proposition
2.8]{igp_survey} and \cite[Theorem 
1.2]{koeppe-zhou:discontinuous-facets}), and we will use them
interchangably.

\smallbreak
To see that the infinite group problem is not a master problem via pullbacks
of homomorphisms, take any homomorphism $\Homo\colon G \to H$, where $H$ is a finite group. 
Then its kernel $K$ is dense in $G$, and thus there is no continuous minimal
function that factors through $\Homo$.




\subsection{Injections of valid functions into finite and infinite group
  problems by extensions}


Instead of homomorphic pullbacks, in the present paper, we wish to discuss a
different, more delicate family of injections of valid inequalities into
larger problems that are right inverses of restrictions.  

Let $P \subset \hat P$ be a chain of subsets of $\R^k/\Z^k$ and let $f \in
\R^k/\Z^k$, $f \neq 0$.
The \emph{restriction} $\pi = \hat\pi|_{P}$ of a valid function
$\hat\pi \colon \hat P\to \R_+$ for $R_f(\hat P)$ is a valid function for $R_f(P)$.
If $\pi\colon P\to\R_+$ is a valid function for $R_f(P)$, then we call a valid
function $\hat\pi\colon \hat P\to\R_+$ for $R_f(\hat P)$ an \emph{extension}
of~$\pi$ to $\hat P$ if $\hat\pi|_P = \pi$.  





Gomory \cite{gom} proved the following theorem, which---together with the
trivial injection \eqref{eq:primal-injection} of solutions---gives the
justification for calling the problems $R_f(G)$, where $G$ is a finite group,
\emph{master problems}.
\begin{theorem}[{Gomory's master theorem for facets, \cite[Theorem 13]{gom}; see also
    \cite[Theorem 19.19]{Richard-Dey-2010:50-year-survey}}]
  \label{thm:gomory-master}
  All facets $\pi$ of a finite group problem $R_f(P)$, where $P \subset G$,
  $f\in G\setminus\{0\}$, and
  $G \subset \Q^k/\Z^k$ is a finite group, arise from facets of any master
  problem $R_f(\hat G)$, where $\hat G$ is a finite group with $G \subseteq
  \hat G$, by restriction.
\end{theorem}
(However, though every restriction of a facet of $R_f(\hat G)$ is a valid
function for $R_f(P)$, it is not necessarily a facet of $R_f(P)$.)  \smallbreak

In the present paper, \textbf{we show that this theorem extends to the 1-row infinite
  group problem}, i.e., $\hat G = \R/\Z$. We have the following theorem.
\begin{theorem}[Infinite master theorem for facets]
  \label{th:infinite-master-for-facets}
  All facets $\pi$ of a finite group problem $R_f(P)$, where $P \subset G$,
  $f\in G\setminus\{0\}$, and $G \subset \Q/\Z$ is a finite group, arise from
  facets of the infinite group problem $R_f(\R/\Z)$ by restriction.  (The
  facets of the infinite group problem can be chosen as continuous piecewise
  linear two-slope functions with rational breakpoints.)
\end{theorem}
Thus, for each fixed right-hand side $f$, \textbf{Gomory--Johnson's infinite
  group problem is the correct ``universal'' master problem for all cyclic
  group relaxations.}
\smallbreak

\begin{openquestion}
  Does \autoref{th:infinite-master-for-facets} generalize from cyclic group
  problems and $\R/\Z$ to arbitrary finite group problems and $\R^k/\Z^k$?
\end{openquestion}
\smallbreak

Before we explain the specific extension of the present paper that
proves \autoref{th:infinite-master-for-facets}, we review related constructions.

For continuous piecewise linear valid functions for $R_f(\R/\Z)$ with rational
breakpoints, 
Gomory and Johnson \cite{infinite,infinite2} proved the following
theorem.
\begin{theorem}[{see \cite[Theorem 19.23]{Richard-Dey-2010:50-year-survey} or
    \cite[Theorem 8.3]{igp_survey_part_2}}]
  \label{thm:restriction-to-cyclic-including-breakpoints-extreme}
Let $\pi$ be an extreme function for $R_f(\R/\Z)$ that is continuous and
  piecewise linear with breakpoints in a cyclic subgroup $G = \frac1q\Z/\Z$%
  .\footnote{The hypothesis regarding the breakpoints
    cannot be removed.  \cite[Theorem 8.2, part (2)]{igp_survey_part_2} (as
    well as the claim that $R_f(G)$ is a face of $R_f(\hat G)$), which seems
    to imply otherwise, is wrong as stated.  We provide a counterexample in
    \autoref{s:restriction-extreme-not-extreme}.
    }  Then the restriction
  $\pi|_G ={}\sagefunc{restrict_to_finite_group}(\pi, q)$ is an extreme
  function for $R_f(G)$. 
\end{theorem}

In the situation of this theorem, $\pi$ is a particular extension of $\pi|_G$,
the \emph{interpolation extension}, 
$\pi = \sagefunc{interpolate_to_infinite_group}(\pi|_G)$.  
Gomory and Johnson \cite[Theorem~3.1]{infinite} 
proved that the interpolation
extension $\pi$ of a minimal function~$\pi|_G$ is valid for $R_f(\R/\Z)$.
In fact, it is also minimal (see, e.g., \cite[section
19.5.1.2]{Richard-Dey-2010:50-year-survey}).  However, 
interpolations of extreme functions $\pi|_G$ to the infinite group problem are
not necessarily extreme.  Dey et al.~\cite{dey1} gave an 
example,
\sage{drlm\_not\_extreme\_1} (see \cite[Figure 15]{igp_survey_part_2}),
illustrating this fact.  Basu et al.~\cite{igp_survey_part_2}
proved the following characterization.
\begin{theorem}[{\cite[Theorem 8.6]{igp_survey_part_2}}, rephrased]
  Let $G = \frac1q\Z/\Z$ and let $\pi|_G$ be an extreme function for
  $R_f(G)$. Let $m \geq 3$ and let $\hat G = \frac1{mq}\Z/\Z$.  Then the
  interpolation extension $\pi$ of $\pi|_G$ is extreme for $R_f(\R/\Z)$ if and
  only if the restriction $\pi|_{\hat G}$ is extreme for $R_f(\hat G)$.
\end{theorem}
(The case $m=4$ appeared in~\cite{basu-hildebrand-koeppe:equivariant}.)
Hence \sagefunc{interpolate_to_infinite_group} does not provide a suitable
injection for proving \autoref{th:infinite-master-for-facets}.


\smallbreak

Gomory--Johnson \cite[Theorem 3.3]{infinite} (with further developments by
Johnson \cite{johnson}) introduced the \emph{two-slope fill-in} procedure,
which constructs a subadditive valid function
$\pi_{\textrm{\textup{fill-in}}} = \sage{two\_slope\_fill\_in}(\pi|_G)$ via
certain sublinear (gauge) functions that arise via the connection to the
mixed-integer infinite relaxation.  $\pi_{\textrm{\textup{fill-in}}}$ is
continuous and piecewise linear with two slopes.  However,
$\pi_{\textrm{\textup{fill-in}}}$ is not symmetric.  Therefore, this extension
does not provide a suitable injection for proving
\autoref{th:infinite-master-for-facets}.

\smallbreak
Recently, Basu--Hildebrand--Molinaro
\cite{bhm:dense-2-slope,bhm:dense-2-slope-fullpaper} proved the following
theorem.
\begin{theorem}[{\cite[Theorem 2]{bhm:dense-2-slope-fullpaper}}]
  Let $\pi$ be a continuous minimal valid function for $R_f(\R/\Z)$, 
  where $f\in\Q/\Z$.
  Let $\epsilon>0$.  Then there exists a
  function~$\pi_{\mathrm{sym}} = \sage{symmetric\underscore 2\underscore slope\underscore
    fill\underscore in}(\pi)$
  with the following properties:
  \begin{enumerate}[\rm(i)]
  \item $\pi_{\mathrm{sym}}$ is continuous and piecewise linear with rational breakpoints
    and 2 slopes,
  \item $\pi_{\mathrm{sym}}$ is extreme for $R_f(\R/\Z)$,
  \item $\|\pi_{\mathrm{sym}}-\pi\|_{\infty} \leq \epsilon$.
  \end{enumerate}
\end{theorem}
Basu et al.'s procedure works as follows, if it is given a minimal $\pi$ that
is already piecewise linear; for our purposes, we would apply it to the
interpolation $\pi = {\sagefunc{interpolate_to_infinite_group}}(\pi|_G)$.
\begin{enumerate}
\item \label{bhm-step-pi_comb}
  Construct a continuous piecewise linear 
  approximation
  $\pi_{\mathrm{comb}}$ that is a minimal function and strongly subadditive,
  $\Delta\pi(x,y)>\gamma>0$, outside of some $\delta$-neighborhood of the trivial additive
  relations $(x,0)$, $(0,y)$ and the symmetry relations $(x, f-x)$.  
\item \label{bhm-step-pi_fill-in}
  Then
  $\pi_{\textrm{\textup{fill-in}}} =
  \sagefunc{two_slope_fill_in}(\pi_{\mathrm{comb}}|_{\hat{G}})$
  where $\hat G = \frac1{\hat q}\Z/\Z$ for some sufficiently large $\hat q$.  
  This gives a subadditive, but not symmetric function
  $\pi_{\textrm{\rm fill-in}}$ that is piecewise linear with two slopes.
\item \label{bhm-step-pi_sym}
  Finally, define $\pi_{\mathrm{sym}}$ as a ``symmetrization'' of
  $\pi_{\textrm{fill-in}}$. This last step crucially depends on step~(\ref{bhm-step-pi_comb}), and
  on specific parameter choices made in (\ref{bhm-step-pi_comb}) and
  (\ref{bhm-step-pi_fill-in}), to make sure that
  symmetrization does not destroy subadditivity.  Then  $\pi_{\mathrm{sym}}$
  is a piecewise linear minimal function with two slopes and hence, by Gomory--Johnson's Two
  Slope Theorem \cite{infinite}, an extreme function.
\end{enumerate}
We illustrate the construction on an example 
in
\autoref{fig:phi-and-pi-sym} (right).  Note that the additivity-reducing
step~(\ref{bhm-step-pi_comb}) may modify the values of $\pi(x)$ for 
$x\in\hat G$, and therefore this procedure does not define an extension of
$\pi|_G$.  
The same is true for the more general construction given by Basu and Lebair 
in \cite{lebair-basu:approximation} for the $k$-dimensional case.

\begin{figure}[tp]
\includegraphics[width=.4\textwidth]{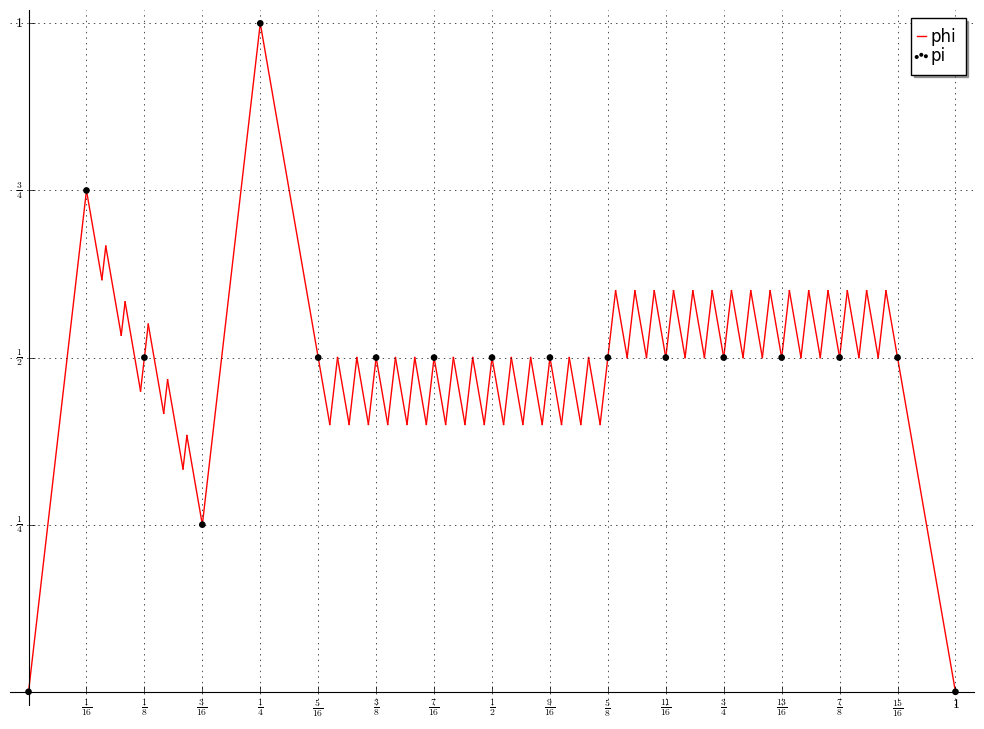}
\includegraphics[width=.4\textwidth]{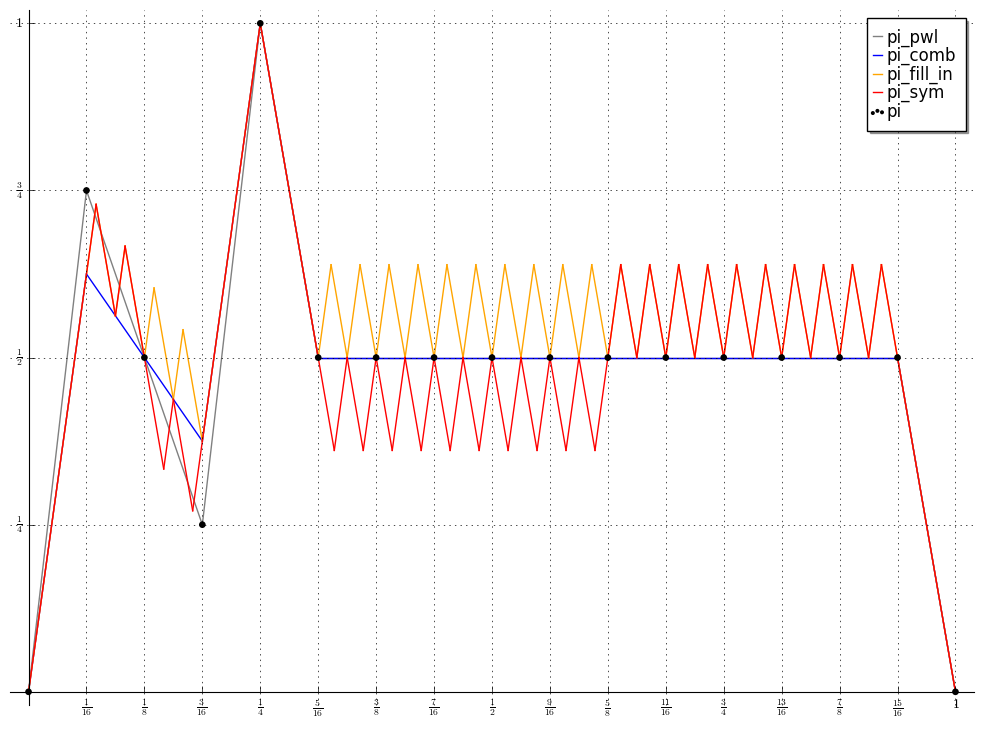}
\caption{Graphs of the approximating extreme functions (left)
  $\phi=\sage{injective\_2\_slope\_fill\_in}(\pi)$ from the present paper
  and (right)
  $\pi_{\text{sym}}=\sage{symmetric\underscore 2\underscore slope\underscore
    fill\underscore in}(\pi)$ from Basu--Hildebrand--Molinaro
  for an example function~$\pi$.  As noted in the introduction, in contrast to
  our approximation $\phi$, Basu et al.'s approximation
  $\pi_{\text{sym}}$ does not preserve the function values on the group
  $\frac{1}{q}\Z/\Z$.}
\label{fig:phi-and-pi-sym}
\end{figure}

\subsection{Technique: Injective approximation of minimal functions by extreme
  functions} 

To prove \autoref{th:infinite-master-for-facets}, we introduce the following variant of
Basu--Hildebrand--Molinaro's approximation; see \autoref{fig:phi-and-pi-sym} (left).

\begin{theorem}[Injective approximation theorem]
  \label{thm:approx-theorem}
  Let $\pi$ be a minimal valid function for $R_f(\R/\Z)$ that is continuous
  and piecewise linear with breakpoints in $\frac{1}{q}\Z/\Z$ and takes only
  rational values at the breakpoints.  There exist integers $r = r_\pi$ and
  $D=D_\pi$ and a function $m_\pi(\epsilon) = O(\epsilon^{-1})$ so that for
  every $\epsilon>0$ and every integer $m \geq m_\pi(\epsilon)$ that is an
  integer multiple of~$r$,
  there exists a function $\phi = \sage{injective\_2\_slope\_fill\_in}(\pi)$
  with the following properties:
  \begin{enumerate}[\rm(i)]
  \item $\phi$ is continuous and piecewise linear with
    breakpoints in $\frac{1}{Dmq}\Z/\Z$ and 2 slopes, 
  \item $\phi$ is extreme for $R_f(\R/\Z)$,
  \item $\|\phi-\pi\|_{\infty} \leq \epsilon$,
  \item $\phi|_{\frac1{mq}\Z/\Z} = \pi|_{\frac1{mq}\Z/\Z}$.
  \end{enumerate}
\end{theorem}

This variant may also be of independent interest.  In particular, it makes a
contribution toward the question regarding the extremality of limits of
extreme functions discussed in \cite[section 6.2]{igp_survey_part_2}.  While
minimality is preserved by taking limits, Dey et al.~\cite[section 2.2,
Example 2]{dey1} constructed a sequence of continuous piecewise linear extreme
functions
that converges pointwise to a
non-extreme, discontinuous piecewise linear 
function.
Basu et al.\
\cite[section 6.2]{igp_survey_part_2} constructed a sequence of continuous
piecewise linear extreme functions of type \sagefuncwithgraph{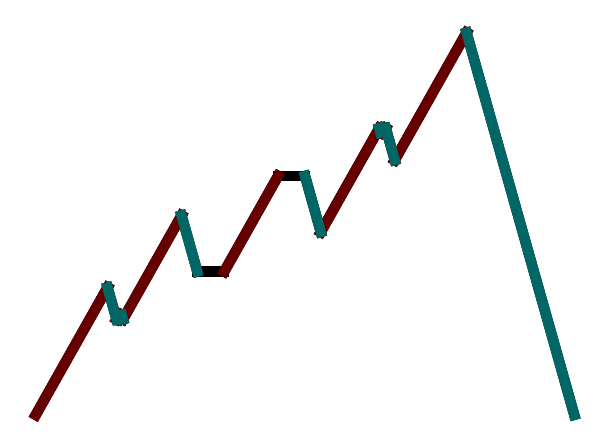} with
irrational breakpoints that converges uniformly to a non-extreme, continuous
piecewise linear function with rational breakpoints.  Our approximation
theorem implies that extremality is not preserved either under another strengthening
of uniform convergence, even for continuous piecewise linear functions with
rational breakpoints. 
\begin{corollary}
  \label{cor:non-extreme-limit}
  For every minimal non-extreme continuous piecewise linear function $\pi$ for $R_f(\R/\Z)$
  with rational breakpoints and rational values at the breakpoints
  there exists a sequence $\{ \phi_i \}_{i\in\N}$ of
  continuous piecewise linear extreme functions for $R_f(\R/\Z)$ with rational
  breakpoints that converges uniformly to~$\pi$ such that
  \begin{equation*}
    \text{for each $x\in\Q/\Z$, the
      sequence $\{ \phi_i(x) \}_{i\in\N}$ is eventually constant.}
\end{equation*}

\end{corollary}

\subsection{Structure of the paper}

The construction and the proof of the approximation theorem
(\autoref{thm:approx-theorem}) appear in
sections~\ref{s:preliminaries}--\ref{s:proof}. Our construction is direct and
avoids the use of an additivity-reducing perturbation (step (1) in Basu et
al.'s approximation, yielding~$\pi_{\mathrm{comb}}$).  As a result, our proof
needs to analyze various cases of the structure of the additivities of~$\pi$
to verify that the subadditivity is not violated, in a way that is similar to
the proof of the finite oversampling theorem
in~\cite{basu-hildebrand-koeppe:equivariant}.  We show some examples
in~\autoref{s:examples}.  Corollaries are proved in
\autoref{s:proofs-corollaries}. 

\section{Preliminaries on approximation}
\label{s:preliminaries}

Let $\pi$ be a continuous minimal valid function for $R_f(\R/\Z)$ that is piecewise
linear with breakpoints in $\frac{1}{q}\Z$ and takes only rational values at
the breakpoints.  We know that $\pi(x) = 0$ for $x = 0$; we can further assume
that $\pi(x) >0$ for any $x \in \R/\Z$, $x \neq 0$, because otherwise $\pi$
(similar to~\autoref{thm:homomorphism}\,(b)) is a
multiplicative homomorphism  of a minimal function that is strictly positive
on $\R/\Z \setminus \{0\}$.

In the remainder of this paper, to simplify notation, we will consider $\pi$
as a $\Z$-periodic function of a real variable.
Given $\epsilon > 0$,
we want to construct a two-slope extreme function $\phi$ such that
$\|\phi - \pi\|_\infty \leq \epsilon$ and $\phi(x) =\pi(x)$ for every
$x \in \frac{1}{q}\Z$.

\begin{lemma}
\label{lemma:approximaition-distance}
Let $\pi$ be a continuous 
piecewise linear function with breakpoints in $\frac{1}{p}\Z$, where $p \in \Z_{>0}$. 
Let $\phi$ be a two-slope continuous piecewise linear function such that $\phi(x) =\pi(x)$ for every $x \in \frac{1}{p}\Z$. Denote the two slope values of $\phi$ by $s^+$ and $s^-$, where $s^+ > s^-$. 
Then \[\|\phi - \pi\|_\infty \leq \frac{s^+-s^-}{4p}.\] 
\end{lemma}
\begin{proof}
The function $\pi$ is affine linear on the interval $[0, \frac{1}{p}]$. Let $s$ be its slope. 
Let $\psi = \phi - \pi$. 
Then, $\psi$ is piecewise linear with slope values $s^+ - s$ and  $s^- - s$ on the interval $[0, \frac{1}{p}]$. 
Let $\ell^+$ and $\ell^-$ denote the Lebesgue measure of the set of $x \in (0, \frac{1}{p})$ with $\psi'(x) = s^+ - s$ and $\psi'(x) = s^- - s$, respectively. We have $\ell^+ + \ell^- = \tfrac{1}{p}$ and $\big( s^+ - s\big)\ell^+ +\big(s^- - s\big)\ell^- = \psi(\tfrac{1}{p})-\psi(0) = 0$, since $\psi(\tfrac{1}{p})=\psi(0)=0$. Hence, \[\ell^+ =  \frac{1}{p} \left(\frac{s-s^-}{s^+-s^-}\right)  \quad \text{ and } \quad \ell^- = \frac{1}{p} \left(\frac{s^+-s}{s^+-s^-}\right).\] 
For $x \in [0, \frac{1}{p}]$, we have that \[|\psi(x)| \leq \max \big\lbrace(s^+-s)\ell^+, (s-s^-)\ell^- \big\rbrace=\frac{(s^+-s)(s - s^-)}{(s^+-s^-)p} \leq \frac{s^+-s^-}{4p}.\]
The same proof works for any interval $[\frac{i}{p}, \frac{i+1}{p}]$, where $i \in \Z$. Therefore, $\|\phi - \pi\|_\infty \leq (s^+-s^-)/{4p}$.
\end{proof}

\section{Construction of the approximation}
\label{s:construction}

Let $s^+$ and $s^-$ denote the slopes of $\pi$ to the right of $0$ and to the left of $1$, respectively. It follows from the subadditivity of $\pi$ that $s^+$ is the most positive slope and $s^-$ is the most negative slope of $\pi$.  

Let $r = r_\pi$ be a positive integer such that $\pi$ satisfies that
\begin{enumerate} 
\item $\frac{f}{2}, \frac{f+1}{2} \in \frac{1}{rq} \Z$, and
\item if $\pi(x) =\frac{1}{2}$, then either $x \in \frac{1}{rq}\Z$ or $\frac{i}{rq} < x < \frac{i+1}{rq}$ for some $i \in \Z$ with $\pi(\frac{i}{rq})=\pi(\frac{i+1}{rq})=\frac{1}{2}$.
\end{enumerate}  
Such integer $r$ exists because $\pi$ takes only rational values at the breakpoints in $\frac{1}{q}\Z$ by assumption. 
Throughout the paper we will denote intervals of length $\frac1{rq}$ with endpoints in
$\frac1{rq}\Z$ by $I'$, $J'$, $K'$. These intervals form a complex~$\P'$, over
which $\pi$ is piecewise linear.

We define 
\begin{equation}
\label{eq:def-delta}
\delta = \min \{\Delta\pi(x,y) \st x,y \in \tfrac{1}{rq}\Z \text{  and } \Delta\pi(x,y) > 0 \}.
\end{equation}
Let $m \in \Z$ such that 
\begin{equation}
\label{eq:def-n}
m \geq m_\pi(\epsilon) := \frac{3(s^+-s^-)}{4 q \min\{3\epsilon, \delta, 1\}} \quad \text{ and
} \quad m=rn \text{\quad for some integer $n$}.
\end{equation}
From these definitions, it follows that the function $\pi$ satisfies the following property.
\begin{prop}
\label{prop:affine-same-sign}
For every $i \in \Z$, the function $\pi$ is affine linear over $[\frac{i}{rq}, \frac{i+1}{rq}]$, and the sign of $\pi(x)-\frac{1}{2}$ is constant for all $x \in (\frac{i}{rq}, \frac{i+1}{rq})$.  In particular, if $\pi(x)=\frac{1}{2}$ and $x \not\in \frac{1}{rq}\Z$, then $\pi'(x) = 0$. 
\end{prop}

For $x\in \R$, 
we use $\{x\}$ to denote the fractional part of $x$, i.e., $\{x\} \in [0,1)$ such that $\{x\} \equiv x \pmod 1$. Let $p$ be a positive integer. We use $\{x\}_{\frac{1}{p}}$ to denote the unique number $y \in [0, \frac{1}{p})$ with $y \equiv x \pmod {\frac{1}{p}}$.  We also define $\lfloor x\rfloor_{\frac{1}{p}} := x- \{x\}_{\frac{1}{p}}$ to be the largest $y \in \tfrac{1}{p}\Z$ such that $y \leq x$, and $\lceil x\rceil_{\frac{1}{p}}$ to be the smallest $y \in \tfrac{1}{p}\Z$ such that $y \geq x$.
 
We define the function $\phi$ as follows. 
Let $x\in \R$. If $\{x\}_{\frac{1}{mq}}=0$, then we set $\phi(x) = \pi(x)$. Otherwise, we define $\phi(x)$ by further distinguishing two cases:
\medskip

\noindent
Case (\CaseM): If $\pi(x) < \frac{1}{2}$ or if $\pi(x) = \frac{1}{2}$ and $\{x\} \not\in [\frac{f}{2},\frac{f+1}{2}]$, then
\begin{equation}
\label{eq:i}
\tag{\CaseM}
\phi(x) =  \pi(x)+ \min \Big\lbrace \big(s^+ - \pi'(x)\big)\{x\}_{\frac{1}{mq}}, \; \big(\pi'(x)-s^-\big)\left(\tfrac{1}{mq} - \{x\}_{\frac{1}{mq}}\right) \Big\rbrace;
\end{equation}

\noindent
Case (\CaseW): If $\pi(x) > \frac{1}{2}$ or if $\pi(x) = \frac{1}{2}$ and $\{x\} \in [\frac{f}{2},\frac{f+1}{2}]$, then
\begin{equation}
\label{eq:ii}
\tag{\CaseW}
\phi(x) =  \pi(x)- \min \Big\lbrace \big(\pi'(x)-s^-\big)\{x\}_{\frac{1}{mq}}, \; \big(s^+ - \pi'(x)\big)\left(\tfrac{1}{mq} - \{x\}_{\frac{1}{mq}}\right) \Big\rbrace.
\end{equation}
For $x \in \R \setminus \frac{1}{mq}\Z$,  define 
\begin{equation}
\label{eq:def-d}
d^+(x) = \frac{1}{mq} \left(\frac{\pi'(x)-s^-}{s^+-s^-}\right) \quad \text{ and } \quad  d^-(x) =\frac{1}{mq} \left(\frac{s^+-\pi'(x)}{s^+-s^-}\right).
\end{equation}
Since $\pi$ only takes rational values on breakpoints, we find $D=D_\pi \in\Z$ so that 
$$\frac{\pi'(x)-s^-}{s^+-s^-} \quad \text{ and } \quad \frac{s^+-\pi'(x)}{s^+-s^-}$$ 
are integer multiples of $\frac1D$ for every $x \in
\R\setminus{\frac1q\Z}$. Then $d^+(x)$ and $d^-(x)$ are integer multiples of
$\frac{1}{Dmq}$.

Concretely, the function $\phi \colon \R \to \R$ constructed above is given by
\begin{equation*}
 \phi(x)  = 
  \begin{cases} 
   \pi(x) & \text{if } \{x\}_{\frac{1}{mq}}=0 \\
   
   \pi(x)+\big(s^+ - \pi'(x)\big)\{x\}_{\frac{1}{mq}}     & \text{if }  \pi(x) < \frac{1}{2}  \text{ and } 0<\{x\}_{\frac{1}{mq}}\leq d^+(x) \\ 
   \pi(x)+\big(\pi'(x)-s^-\big)\left(\frac{1}{mq} - \{x\}_{\frac{1}{mq}}\right)    & \text{if } \pi(x) < \frac{1}{2}  \text{ and }  \{x\}_{\frac{1}{mq}}> d^+(x) \\
   
    \pi(x)-\big(\pi'(x)-s^-\big)\{x\}_{\frac{1}{mq}}     & \text{if }  \pi(x) > \frac{1}{2}  \text{ and } 0<\{x\}_{\frac{1}{mq}}\leq d^-(x) \\ 
   \pi(x)-\big(s^+-\pi'(x)\big)\left(\frac{1}{mq} - \{x\}_{\frac{1}{mq}}\right)     & \text{if }  \pi(x) > \frac{1}{2}  \text{ and }  \{x\}_{\frac{1}{mq}}> d^-(x) \\ 
   
     \frac{1}{2} + s^+\cdot \{x\}_{\frac{1}{mq}}   & \text{if } \pi(x)=\frac{1}{2} \text{ and } \{x\} \not\in [\frac{f}{2},\frac{f+1}{2}] \\
     &\text{ and }0< \{x\}_{\frac{1}{mq}}\leq  d^+(x) \\ 
     \frac{1}{2} - s^-\cdot\left(\frac{1}{mq} - \{x\}_{\frac{1}{mq}}\right)  & \text{if } \pi(x)=\frac{1}{2} \text{ and } \{x\} \not\in [\frac{f}{2},\frac{f+1}{2}] \\
     &\text{ and } \{x\}_{\frac{1}{mq}}> d^+(x) \\ 
     
     \frac{1}{2} + s^-\cdot \{x\}_{\frac{1}{mq}}   & \text{if } \pi(x)=\frac{1}{2} \text{ and } \{x\} \in [\frac{f}{2},\frac{f+1}{2}] \\
     & \text{ and } 0< \{x\}_{\frac{1}{mq}}\leq  d^-(x) \\ 
     \frac{1}{2} - s^+\cdot \left(\frac{1}{mq} - \{x\}_{\frac{1}{mq}}\right)  & \text{if } \pi(x)=\frac{1}{2} \text{ and } \{x\} \in [\frac{f}{2},\frac{f+1}{2}] \\
     & \text{ and } \{x\}_{\frac{1}{mq}}>  d^-(x) 
  \end{cases}
\end{equation*}

\section{Proof of Theorem~\ref{thm:approx-theorem}}
\label{s:proof}

\begin{lemma}
\label{lemma:phi-prop}
The function $\phi$ defined above has the following properties.
\begin{enumerate}
\item $\phi$ is $\Z$-periodic.
\item $\phi(x)=0$ for all $x \in \Z$.
\item $\phi$ satisfies the symmetry condition: $\phi(x)+\phi(f-x)=1 $ for all $x\in \R$.
\item $\phi(\frac{i}{mq}) = \pi(\frac{i}{mq})$ for every $i \in \Z$.
\item $\phi$ is continuous piecewise linear and has two distinct slopes values $s^+$ and $s^-$.
\item For any integer $i$, the lengths of the subintervals of $(\frac{i}{mq}, \frac{i+1}{mq})$ where $\phi$ takes positive slope $s^+$ and negative slope $s^-$ are $d^+(x)$ and $d^-(x)$, respectively, where $x$ is any number in $(\frac{i}{mq}, \frac{i+1}{mq})$.
\item 
$\phi$ has its breakpoints in $\frac1{Dmq}\Z$.
\item For every $i\in\Z$, $\phi(x)$ is either given by equation \eqref{eq:i} for
  all $x \in [\frac{i}{rq}, \frac{i+1}{rq}]$ or by equation~\eqref{eq:ii} for
  all $x \in [\frac{i}{rq}, \frac{i+1}{rq}]$.
  In particular, $\phi(x) - \pi(x)$ is periodic with period $\frac1{mq}$ on each
  interval $[\frac{i}{rq}, \frac{i+1}{rq}]$.
  \label{lemma:phi-prop:quasiperiodic-on-rq}
\end{enumerate}
\end{lemma}
\begin{proof}
The properties (1) to (4) follow directly from the definition of $\phi$ and the fact that $\pi$ is a minimal valid function.
Let $i \in \Z$. 
Recall that 
the function $\pi$ is affine linear on the interval $(\frac{i}{mq}, \frac{i+1}{mq})$. Thus, $d^+$ and $d^-$ are constant functions on the interval $(\frac{i}{mq}, \frac{i+1}{mq})$. Let $x \in (\frac{i}{mq}, \frac{i+1}{mq})$. Then, $\{x\}_{\frac{1}{mq}} = x - \frac{i}{mq} > 0$, $\lfloor x \rfloor_{\frac{1}{mq}} = \frac{i}{mq}$ and $\lceil x \rceil_{\frac{1}{mq}} = \frac{i+1}{mq}$. We have 
\[d^+(x)+d^-(x) = \tfrac{1}{mq}.\] 
It is clear that in Case \eqref{eq:i}, the value $\phi(x)$ satisfies
\begin{equation}
\tag{\CaseM$'$}
\label{eq:phi-case-i}
\phi(x) = 
 \begin{cases} 
 \pi(\lfloor x \rfloor_{\frac{1}{mq}}) + s^+ \cdot \{x\}_{\frac{1}{mq}}  & \text{if } \{x\}_{\frac{1}{mq}}  \leq d^+(x) \\
 \pi(\lceil x \rceil_{\frac{1}{mq}} ) - s^- \cdot \big(\tfrac{1}{mq}-\{x\}_{\frac{1}{mq}}\big) & \text{otherwise}; 
 \end{cases}
\end{equation}
and in Case \eqref{eq:ii}, the value $\phi(x)$ satisfies
\begin{equation}
\tag{\CaseW$'$}
\label{eq:phi-case-ii}
\phi(x) = 
 \begin{cases} 
 \pi(\lfloor x \rfloor_{\frac{1}{mq}}) + s^- \cdot \{x\}_{\frac{1}{mq}} & \text{if } \{x\}_{\frac{1}{mq}} \leq d^-(x) \\
 \pi(\lceil x \rceil_{\frac{1}{mq}}) - s^+ \cdot \big(\tfrac{1}{mq}-\{x\}_{\frac{1}{mq}}\big) & \text{otherwise}. 
 \end{cases}
\end{equation}
Therefore, the properties (5) and (6) follow. 
Property~(7) follows from the definition of $d^+(x)$, $d^-(x)$ and $D$. 
Finally, property~(\ref{lemma:phi-prop:quasiperiodic-on-rq}) follows from the
definition of~$r$ together with \autoref{prop:affine-same-sign}.
\end{proof}

\begin{theorem}
\label{thm:phi-approximates}
The function $\phi$ defined above is an approximation of $\pi$, such that $\|\phi-\pi\|_{\infty} \leq \epsilon$.
\end{theorem}
\begin{proof}
  By \autoref{lemma:phi-prop} (4)--(5), $\phi$ satisfies the hypotheses of
  \autoref{lemma:approximaition-distance} with $p = mq$.  
  Thus, we have that  $\|\phi-\pi\|_{\infty} \leq \frac{s^+-s^-}{4mq}$. Since $m \geq \frac{s^+-s^-}{4q\epsilon}$ by \eqref{eq:def-n}, the result follows. 
\end{proof}

\begin{lemma}
\label{lemma:phi-non-negative}
The function $\phi$ defined above is non-negative.
\end{lemma}
\begin{proof}
Let $x \in \R$. If $\pi(x) < \frac{1}{2}$, then by the definition of $\phi$, we have that $\phi(x) \geq \pi(x) \geq 0$. Now assume that $\pi(x) \geq \frac{1}{2}$. Let $p = mq$. Then $p \geq \frac{3}{4}(s^+-s^-)$ by \eqref{eq:def-n}. It follows from \autoref{lemma:approximaition-distance} that $|\phi(x) -\pi(x)| \leq \frac{1}{3}$. Therefore, $\phi(x) \geq 0$. The function $\phi$ is non-negative.
\end{proof}

\begin{lemma}
\label{lemma:phi-subadditive}
The function $\phi$ defined above is subadditive.
\end{lemma}
\begin{proof}
Since $\phi$ is a $\Z$-periodic function by \autoref{lemma:phi-prop}-(1), it suffices to show that $\Delta\phi(x,y)\geq 0$ for any $x, y \in [0,1)$. 
If $x=0$ or $y=0$, then we have that $\Delta\phi(x,y) = 0$, since $\phi(0) =0$ by \autoref{lemma:phi-prop}-(2).

It is known from \cite{basu-hildebrand-koeppe:equivariant} that the square
$[0,1]^2$ is tiled by the lower and upper triangles%
\begin{equation}
  \label{eq:face-f-prime}
  F' = F(I',J',K') = \big\lbrace\,
  (x,y) \in [0,1]^2 \,\st\, 
  x \in I', \; y \in J', \; x + y \in K'\,\big\rbrace,
\end{equation}
where $I', J', K'$ are closed proper intervals of length $\frac1{rq}$ with
endpoints in $\frac1{rq}\Z$.  We note that $\pi$ is affine linear over each of
the intervals, and thus $\Delta\pi$ is affine linear over $F'$.  The triangles
$F'$ are the maximal faces of a polyhedral complex $\Delta\P'$.

It suffices to prove that $\Delta\phi$ is nonnegative over all lower
triangles: The case of upper triangles can be transformed by the map
$(x, y) \mapsto -(x, y)$, since $\pi$ being a minimal valid function for
$R_{f}(\R/\Z)$ implies that the function $x \mapsto \pi(-x)$ is minimal valid
for $R_{1-f}(\R/\Z)$, and the function $\phi$ is covariant under this map as
well.\smallbreak

In the following, we fix a lower triangle $F' = F(I', J', K')$. 
If $\pi$ is strictly subadditive over all of $F'$, we easily prove the
subadditivity of $\phi$.
\begin{claim}
  \label{claim:phi-subadd-F-strictly-subadd}
  If $\Delta\pi(u,v)> 0$ for every vertex $(u, v)$ of $F'$, then $\Delta\phi(x,y)\geq 0$.
\end{claim}

\begin{proof}
By the definition \eqref{eq:def-delta} of $\delta$, we have that $\Delta\pi(u,v) \geq \delta$ on every vertex $(u,v)$ of $F'$. Then, $\Delta\pi(x,y)\geq \delta$, since $\Delta\pi$ is affine linear over $F'$. Using \autoref{lemma:approximaition-distance} with $p = mq = rnq$ and \eqref{eq:def-n}, we obtain  that $|\phi(x,y)-\pi(x,y)| \leq \frac{\delta}{3}$. It follows that $\Delta\phi(x,y) \geq 0$.
\end{proof}

\DIFFPROTECT{%
Otherwise, there is a vertex $(u, v)$ of $F'$ with $\Delta\pi(u,v)=0$.
The triangle $F'$ is tiled by smaller lower and upper triangles 
\begin{equation}
  \label{eq:face-f}
  F(I,J,K) = \big\lbrace\,
  (x, y) \in [0,1]^2 \,\st\, 
  x \in I, \; y \in J, \; x + y \in K\,\big\rbrace,
\end{equation}
where $I, J, K$ are closed proper intervals of length $\frac1{mq}$ with
endpoints in $\frac1{mq}\Z$.  Again the triangles $F(I,J,K)$ form the maximal faces
of a polyhedral complex $\Delta\P$, which is a refinement of $\Delta\P'$.
Moreover, 
by
\autoref{lemma:phi-prop}-(\ref{lemma:phi-prop:quasiperiodic-on-rq}), $\phi -
\pi$ is periodic with period $\frac1{mq}$ on each of $I'$, $J'$, and $K'$. 
Thus, within $F'$, the function $\Delta\phi-\Delta\pi$ is invariant under translations by $(s,t)$ where
$s, t \in \frac1{mq}\Z$:
$$\Delta\phi(x,y) - \Delta\pi(x,y) 
= \Delta\phi(x + s, y + t) - \Delta\pi(x + s, y + t).$$
This implies the following.
\begin{observation}\label{obs:shifting-back-in-tangent-cone}
  Let $(u, v)$ be a vertex of $F'$ with $\Delta\pi(u,v)=0$.
  Let $F = F(I,J,K) \subseteq F'$ be a triangle as in~\eqref{eq:face-f}
  and $s, t \in \frac1{mq}\Z$ such that
  $(u, v) + (s,t) \in F'$ 
  and $F + (s,t) \subseteq F'$. 
  Then for $(x, y) \in F$, 
  \begin{equation}\label{eq:shifting-back-in-tangent-cone}
    \begin{aligned}
      \Delta\phi(x+s, y+t) &= \Delta\phi(x, y) + \Delta\pi(x+s,y+t) - \Delta\pi(x,y) \\
      &= \Delta\phi(x, y) + \Delta\pi(u+s, v+t) - \Delta\pi(u, v) \\
      &= \Delta\phi(x, y) + \Delta\pi(u+s, v+t).
    \end{aligned}
  \end{equation}
\end{observation}}

Because $\Delta\pi\geq0$ in \eqref{eq:shifting-back-in-tangent-cone}, it
will suffice to prove $\Delta\phi(x,y) \geq 0$ when $(x,y)$ lies in 
some canonical triangle $F$ near the additive vertex $(u,v)$.  
We distinguish the following cases; see Figures~\ref{fig:phi-subadd-F-a-aprime}
and \ref{fig:phi-subadd-F-b-bprime}.

\textbf{Case (a).} $(u,v)$ is the lower left corner of the lower triangle
$F'$.  Define the canonical lower triangle $F
= \conv\{(u,v), (u+\frac1{mq},v), (u, v+\frac1{mq})\}$.
Then every lower triangle $F(I,J,K) \subseteq F'$ arises as $F + (s,t)$
according to \autoref{obs:shifting-back-in-tangent-cone}. 

\textbf{Case (a').} $(u,v)$ is the lower left corner of the lower triangle
$F'$.  Define the canonical upper triangle $F 
= \conv\{(u+\frac1{mq},v), (u, v+\frac1{mq}), (u+\frac1{mq}, v+\frac1{mq})\}$. 
Then every upper triangle $F(I,J,K) \subseteq F'$ arises as $F + (s,t)$
according to \autoref{obs:shifting-back-in-tangent-cone}. 

\textbf{Case (b).} $(u,v)$ is the upper left corner of the lower triangle
$F'$.
Define the canonical lower triangle $F
= \conv\{(u,v), (u,v-\frac1{mq}), (u+\frac1{mq}, v-\frac1{mq})\}$.
Then every lower triangle $F(I,J,K) \subseteq F'$ arises as $F + (s,t)$
according to \autoref{obs:shifting-back-in-tangent-cone}. 

\textbf{Case (b').} $(u,v)$ is the upper left corner of the lower triangle
$F'$.  Define the canonical upper triangle $F
= \conv\{(u,v-\frac1{mq}), (u+\frac1{mq}, v-\frac1{mq}), (u+\frac1{mq}, v-\frac2{mq}\}$.
Then every upper triangle $F(I,J,K) \subseteq F'$ arises as $F + (s,t)$
according to \autoref{obs:shifting-back-in-tangent-cone}. 

The case that the additive vertex $(u,v)$ is the lower right corner can be reduced to Cases (b)
and (b'), respectively, since $\Delta\phi$ is invariant under the map
$(x,y) \mapsto (y,x)$.

\begin{figure}[t]
  \centering
  \includegraphics[width=.45\linewidth]{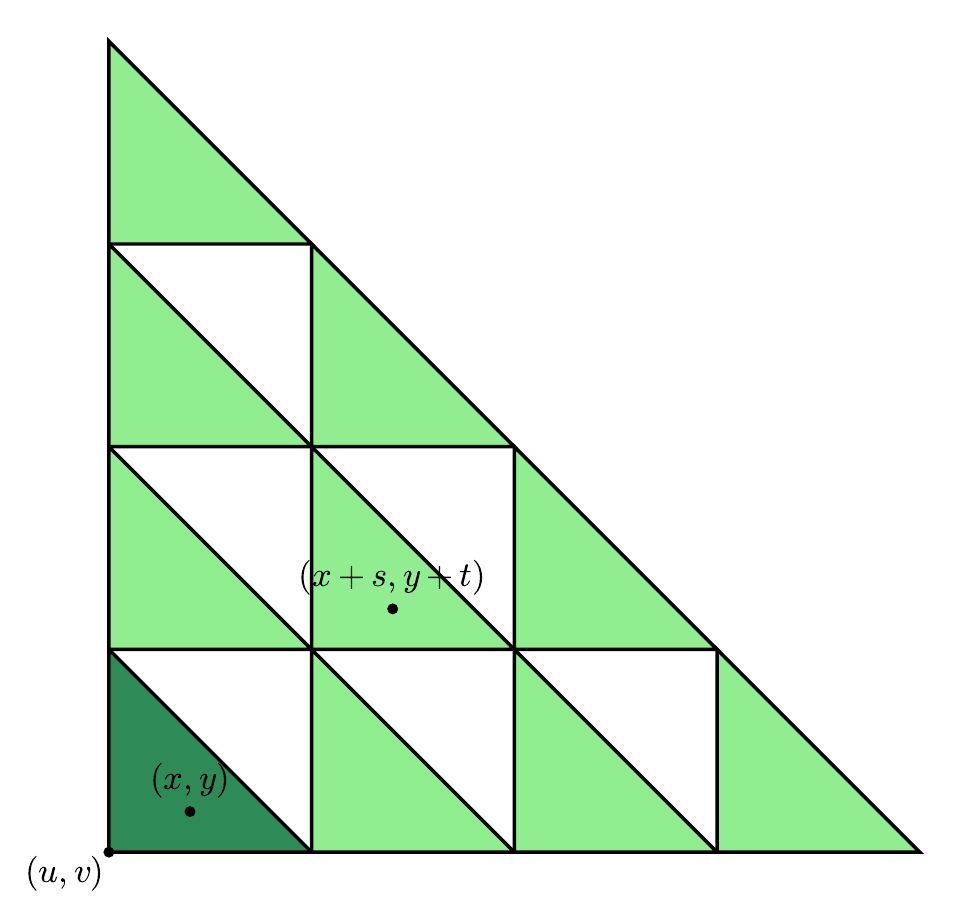}\hfill
  \includegraphics[width=.45\linewidth]{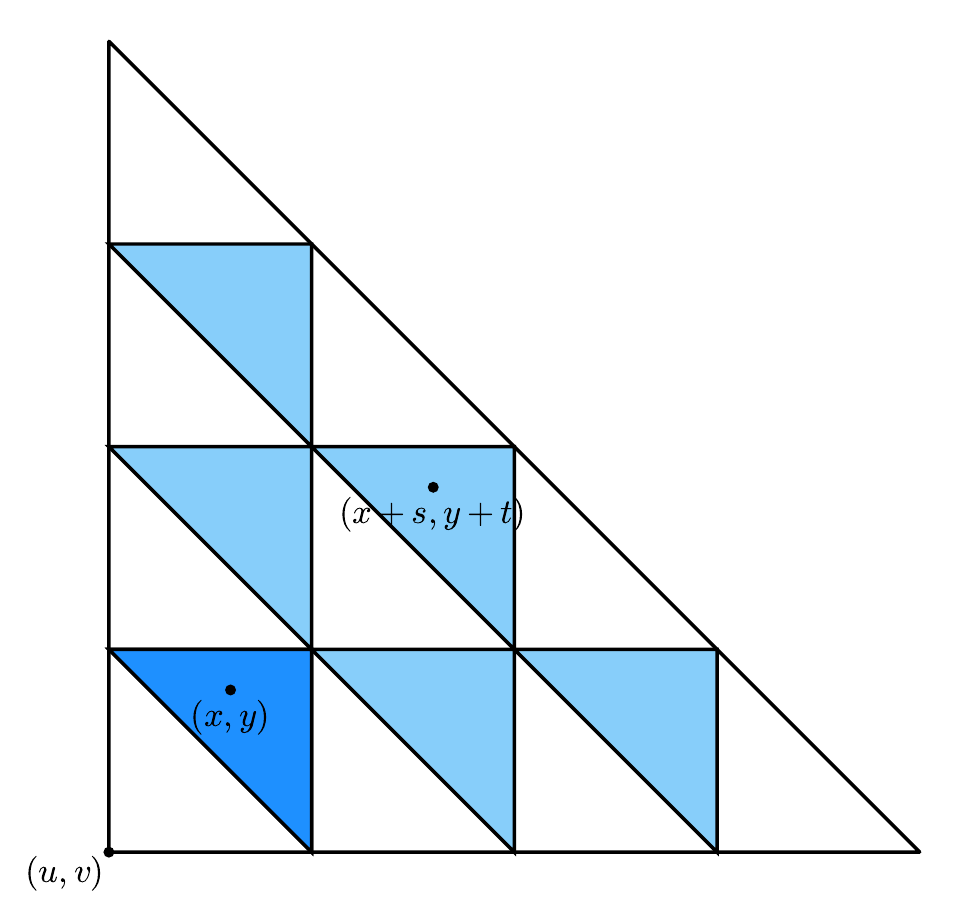}
  \caption{Cases (a) and (a') in the proof of \autoref{lemma:phi-subadditive}
    (Claims \ref{claim:phi-subadd-F-dim-2} and \ref{claim:phi-subadd-F-a'}).
    The canonical triangle~$F$ is \emph{shaded dark}, its translations
    $F+(s,t)$ within the larger triangle $F'$ according
    to \autoref{obs:shifting-back-in-tangent-cone} are \emph{shaded lighter}.
}
  \label{fig:phi-subadd-F-a-aprime}
\end{figure}
\begin{claim}
  \label{claim:phi-subadd-F-dim-2}
  For the canonical lower triangle $F$ of Case~(a), 
  $\Delta\phi(x,y)\geq 0$ for $(x, y) \in F$.
\end{claim}
\begin{proof}
It suffices to prove the claim for $(x, y) \in \intr(F)$; it extends to the
boundary by continuity.  So we have 
\begin{equation}
\label{eq:case-a}
u < x < u +\tfrac{1}{mq}, \quad v < y < v+\tfrac{1}{mq} \quad  \text{ and  } \quad u+v < x+y < u+v+\tfrac{1}{mq}.
\end{equation}
We note that 
$ \{ x \}_{\frac1{mq}} = x - u $, $ \{ y \}_{\frac1{mq}} = y - v$, and $ \{ x
+ y \}_{\frac1{mq}} = x + y - (u+v)$.

From the additivity $\pi(u+v) = \pi(u) + \pi(v)$ and the subadditivity
$\pi(u+v+\frac1{mq}) \leq \pi(u) + \pi(v+\frac1{mq})$ we obtain
$\pi'(x+y) = mq\bigl(\pi(u+v+\frac1{mq}) - \pi(u+v)\bigr) \leq mq\bigl(\pi(v+\frac1{mq}) -
\pi(v)\bigr) = \pi'(y)$. Likewise, $\pi'(x+y) \leq \pi'(y)$ follows.
Therefore, 
\[d^+(x)\geq d^+(x+y) \quad \text{ and }\quad d^+(y)\geq d^+(x+y);\]
\[d^-(x)\leq d^-(x+y) \quad \text{ and }\quad d^-(y)\leq d^-(x+y).\]
By the symmetry condition, we may assume that $\pi(u) \leq \pi(v)$. Since
$\pi(u)+\pi(v)=\pi(u+v)\leq 1$, we have either (a\oldstylenums0)
$\pi(u)=\pi(v)=\frac{1}{2}$, or (a\oldstylenums1--a\oldstylenums4) $\pi(u)<\frac{1}{2}$. 
In subcase (a\oldstylenums0), $\pi(u+v)=1$, implying that $\{u+v\}=f$. We know
that the function $\pi$ has slope $s^-$ on the interval $(f, f+\frac{1}{mq})$,
and so has the function $\phi$. Thus, $\phi(x+y)=\pi(u+v)+s^-\cdot \{ x+y
\}_{\frac{1}{mq}}$. We also know that $\phi(x) \geq \pi(u)+s^-\cdot \{ x
\}_{\frac{1}{mq}}$ and $\phi(y) \geq \pi(v)+s^-\cdot \{ y
\}_{\frac{1}{mq}}$. Therefore, $\Delta\phi(x,y) \geq \Delta\pi(x,y) \geq
0$. In the latter case (a\oldstylenums1--a\oldstylenums4), we know that $\phi(x)$ is given by equation \eqref{eq:phi-case-i} in Case \eqref{eq:i}. If $\phi(y)$ is also given by equation \eqref{eq:phi-case-i} in Case \eqref{eq:i}, then $\phi$ has slope $s^+$ on the intervals $\big( u, u+d^+(x) \big)$ and $\big( v, v+d^+(y)\big)$ with $d^+(x)\geq d^+(x+y)$ and $d^+(y)\geq d^+(x+y)$. 
It follows that
\begin{enumerate}
\item[(a\oldstylenums1)] when $\{ x+y \}_{\frac{1}{mq}} \leq d^+(x+y)$, 
  then the length of the subinterval of $[u+v, x+y]$ where $\phi$ has
  slope~$s^+$ is at most $\{ x + y \}_{\frac{1}{mq}}$ (with equality when $\phi(x+y)$ is obtained by
  Case~\eqref{eq:i}),
  and we can estimate $\phi(x+y) \leq\pi(u+v)+s^+\cdot \{ x+y \}_{\frac{1}{mq}}$. 
  Moreover, $\phi(x)=\pi(u)+s^+\cdot \{ x\}_{\frac{1}{mq}}$ and
  $\phi(y)=\pi(v)+s^+\cdot \{ y \}_{\frac{1}{mq}}$, which imply that
  $\Delta\phi(x,y)\geq \Delta\pi(u,v) = 0$. 
\end{enumerate}
On the other hand,
\begin{enumerate}
\item[(a\oldstylenums2)] when $\{ x+y \}_{\frac{1}{mq}}> d^+(x+y)$, we
  estimate as follows:  First, the
  length of the subinterval of $[u+v, x+y]$ where $\phi$ has slope~$s^+$ is at
  most $d^+(x+y)$ (with equality when $\phi(x+y)$ is obtained by
  Case~\eqref{eq:i}), and hence we have $\phi(x+y) \leq\pi(u+v)+s^+\cdot
  d^+(x+y)+s^- \cdot \big(\{ x+y \}_{\frac{1}{mq}} - d^+(x+y)\big)$.
  Next, we show that the sum of the 
  lengths of the subintervals of $[u,x]$ and $[v,y]$ on which $\phi$ has slope $s^+$
  is at least $d^+(x+y)$.
  If $\{ x \}_{\frac{1}{mq}}>d^+(x) \geq d^+(x+y)$ or $\{ y
  \}_{\frac{1}{mq}}>d^+(y) \geq d^+(x+y)$, this holds. 
  Otherwise, $\phi$ has slope $s^+$ on the whole intervals $[u,x]$ and
  $[v,y]$, 
  so the sum of the lengths is $\{ x \}_{\frac{1}{mq}} + \{ y
  \}_{\frac{1}{mq}} = \{ x + y \}_{\frac{1}{mq}} > d^+(x+y)$. 
  Thus we can estimate $\phi(x)+\phi(y) \geq\pi(u)+\pi(v)+s^+\cdot d^+(x+y)+s^- \cdot \big(\{ x \}_{\frac{1}{mq}} +\{ y \}_{\frac{1}{mq}}  - d^+(x+y)\big)$. Therefore, $\Delta\phi(x,y) \geq 0$ also holds.
\end{enumerate}
Now assume that $\phi(y)$ is given by equation \eqref{eq:phi-case-ii} in Case \eqref{eq:ii}. We know that $\pi(v)\geq \frac{1}{2}$. If $\pi(u) = 0$, then $\{u\}=0$. This implies that $\pi$ has slope $s^+$ on the interval $(u, u+\frac{1}{mq})$, and so does $\phi$. We have on the one side $\phi(x) = \pi(u) + s^+\cdot \{ x\}_{\frac{1}{mq}}=s^+\cdot \{ x\}_{\frac{1}{mq}}$. On the other side, $\phi(x+y)-\phi(y) = \phi( x-u+y) -\phi(y) \leq s^+\cdot \{ x\}_{\frac{1}{mq}}$. It follows that $\Delta\phi(x,y) \geq 0$, when $\pi(u)=0$. If  $\pi(u) > 0$, then $\pi(u+v) = \pi(v)+\pi(u) > \frac{1}{2}$, which implies that $\phi(x+y)$ is also given by equation \eqref{eq:phi-case-ii} in Case \eqref{eq:ii}. Notice that $\phi$ has slope $s^-$ on the intervals $\big(v, v+d^-(y)\big)$ and  $\big(u+v, u+v+d^-(x+y)\big)$ with  $d^-(x+y) \geq d^-(y)$, and that $\phi$ has slope $s^+$ on the interval $\big(u, u+d^+(x)\big)$ with $d^+(x) \geq d^+(x+y)$. It follows that
\begin{enumerate}
\item[(a\oldstylenums3)] when $ \{ x+y \}_{\frac{1}{mq}} \leq d^-(x+y)$, we
  have $\phi(x+y)=\pi(u+v)+s^-\cdot  \{ x+y \}_{\frac{1}{mq}}$, and we can
  estimate $\phi(x)+\phi(y) \geq\pi(u) + s^- \cdot  \{ x\}_{\frac{1}{mq}}
  +\pi(v)+s^-\cdot  \{ y\}_{\frac{1}{mq}}$.
\end{enumerate}
On the other hand,
\begin{enumerate}
\item[(a\oldstylenums4)] when $ \{ x+y \}_{\frac{1}{mq}} > d^-(x+y)$, we have
  $\phi(x+y) =\pi(u+v)+s^-\cdot d^-(x+y)+s^+\cdot \big( \{ x+y
  \}_{\frac{1}{mq}} - d^-(x+y)\big)$.
  Next, we show that the sum of the lengths of the subintervals
  of $[u, x]$ and $[v, y]$ on which $\phi$ has slope~$s^-$ is at most $d^-(x+y)$.
  If $\{ x \}_{\frac{1}{mq}} < d^+(x)$, then $\phi$ has slope $s^+$ on all of
  $[u,x]$, and the sum of the lengths has the upper bound $d^-(y) \leq d^-(x+y)$.
  If $\{ x \}_{\frac{1}{mq}} \geq d^+(x) \geq d^+(x+y)$, then $\phi$ has slope
  $s^+$ on $[u, u+d^+(x+y)]$, and the sum of the lengths has the
  upper bound $\{ x \}_{\frac{1}{mq}} + \{ y \}_{\frac{1}{mq}} - d^+(x+y) 
  = \{ x+y \}_{\frac{1}{mq}} - d^+(x+y) \leq \frac1{mq} -  d^+(x+y) = d^-(x+y)$.
  Thus we can estimate
  $\phi(x)+\phi(y) \geq\pi(u)+\pi(v)+s^-\cdot d^-(x+y)+s^+ \cdot \big(\{ x \}_{\frac{1}{mq}} +\{ y \}_{\frac{1}{mq}}- d^-(x+y)\big)$. 
\end{enumerate} 
In either case, $\Delta\phi(x,y)\geq 0$ holds. 
\end{proof}

For Case (a'), because the upper triangle $F$ shares each its edges with
lower triangles, which fall into Case (a), we have that $\Delta\phi \geq 0$ on the boundary of~$F$. 
Consider the arrangement of hyperplanes (lines) $x = b$, $y = b$, and $x + y = b$,
where $b$ runs through the breakpoints of the function~$\phi$.  The cells of
this arrangement define a polyhedral complex $\Delta\P^\phi$, which is a
refinement of the complex~$\Delta\P$.  Then $\Delta\phi$ is affine linear on each
maximal face of $\Delta\P^\phi$.
Thus it suffices to show the following.
\begin{claim}\label{claim:phi-subadd-F-a'}
  Let $F$ be the canonical upper triangle of Case (a').  Then
  $\Delta\phi \geq 0$ on all vertices of the complex $\Delta\P^\phi$ that lie
  in $\intr(F)$.
\end{claim}
\begin{proof}
  \DIFFPROTECT{%
  We have $F = F(I,J,K)$ for $I = [u, u+\frac1{mq}]$, $J = [v, v+\frac1{mq}]$,
  and $K = [u+v+\frac1{mq}, u+v+\frac2{mq}]$. 
  Let $s_1, s_2$ and $s_3$ denote the slopes of $\pi$ on the intervals $I, J$ and $K$, respectively.
  Let $\bx$, $\by$, and $\bz$ denote the unique breakpoints of $\phi$ in $\intr(I)$,
  $\intr(J)$, and $\intr(K)$, respectively.  Then the only vertices
  of $\Delta\P^\phi$ that can lie in $\intr(F)$ are $(\bar x, \bar y)$, $(\bar x, \bar z-\bar x)$, and
  $(\bar z-\bar y, \bar y)$.  The breakpoint $\bar x$ splits $I$ into an interval $I^+$, on which
  $\phi$ has slope $s^+$, and $I^-$, on which $\phi$ has slope $s^-$. Likewise, we
  obtain intervals $J^+, J^-, K^+, K^-$, whose lengths are given by
  \eqref{eq:def-d}.  
  As in the proof above, we have that $s_1 \geq s_3$ and $s_2 \geq s_3$,
  and we may assume $\pi(u) \leq \pi(v)$.
  In the subcase
  (a'\oldstylenums0) $\pi(u) = \pi(v) = \frac12$, the subadditivity is verified
  by the same proof. 
  So we assume $\pi(u) < \frac12$, and thus $\phi(x)$ is given on $I'$ by
  Case \eqref{eq:i}.}
 
  If $\phi(y)$ is also given on $J'$ by Case \eqref{eq:i} and $\phi(x+y)$ is
  given on $K'$ by Case \eqref{eq:ii}, then for any $(x,y) \in F$, we have that
  $\phi(x) \geq \pi(x)$, $\phi(y) \geq \pi(y)$ and $\phi(x+y) \leq \pi(x+y)$. 
  It follows that $\Delta\phi(x,y) \geq \Delta\pi(x,y) \geq 0$. Therefore, the claim holds.
  
  If $\phi(y)$ is also given on $J'$ by Case \eqref{eq:i} and $\phi(x+y)$ is
  given on $K'$ by Case \eqref{eq:i}, 
  then $\bar x = u+|I^+|$, $\bar y = v+|J^+|$ and $\bar z = u+v+\frac1{mq}+|K^+|$.
  We distinguish two subcases (combinatorial types) as follows.
  
  \begin{enumerate}
  \item[(a'\oldstylenums1)] Assume that $\bar x + \bar y \leq \bar z$.
  If the vertex $(\bar x, \bar y) \in \intr(F)$, then $\bar x + \bar y
  >u+v+\frac{1}{mq}$, and we have  
  $\Delta\phi(\bar x, \bar y)  = |K^-| (s^+ - s^-) = \frac1{mq}(s^+ - s_3) \geq 0$ as $\bar x+\bar y \in K^+$.
  Both vertices $(\bar x, \bar z - \bar x)$ and $(\bar z -\bar y, \bar y)$ lie in $F$, 
  since $s_1 \geq s_3$ and  $s_2 \geq s_3$ imply that $|I^+|>|K^+|$ and $|J^+|>|K^+|$.
  Since $\bar x + \bar y \leq \bar z$, 
  we have  $\bar z - \bar y \in I^-$ and $\bar z -\bar x \in J^-$. 
  Thus, $\Delta\phi(\bar x, \bar z -\bar x) = \Delta\phi(\bar z - \bar y, \bar
  y) = (|I^+|+|J^+|-2|K^+|)(s^+ -s^-) = \frac{1}{mq}((s_1 - s_3) + (s_2 - s_3))\geq 0$.
  \item[(a'\oldstylenums2)] Assume that $\bar x + \bar y > \bar z$. All three vertices $(\bar x, \bar y)$, $(\bar x, \bar z-\bar x)$ and
  $(\bar z-\bar y, \bar y)$ lie in $F$. We have $\bar x+\bar y \in K^-$, $\bar z-\bar x  \in J^+$ and $\bar z-\bar y \in I^+$. 
  It follows that $\Delta\phi(\bar x, \bar y) = (|I^+|+|J^+|-2|K^+|)(s^+ -s^-) = \frac{1}{mq}((s_1-s_3)+(s_2 - s_3))\geq 0$ and $\Delta\phi(\bar x, \bar z -\bar x) = \Delta\phi(\bar z - \bar y, \bar y) = |K^-|(s^+ - s^-) = \frac{1}{mq}(s^+ - s_3) \geq 0$.
  \end{enumerate}

 Now assume that $\phi(y)$ is given on $J'$ by Case \eqref{eq:ii}. As in the proof of \autoref{claim:phi-subadd-F-dim-2}, if $\pi(u)=0$, then $\phi(x,y)\geq 0$ follows easily for all $(x,y)\in F$, and if $\pi(u)>0$, then $\phi(x+y)$ is given on $K'$ by Case \eqref{eq:ii}. Then $\bar x =u+|I^+|$, $\bar y =v +|J^-|$ and $\bar z = u+v+\frac{1}{mq}+|K^-|$. 
This is subcase (a'\oldstylenums3).
Since $s_2\geq s_3$,  we have $|J^-|\leq |K^-|$ by equation~\eqref{eq:def-d}. Also, $|I^+| \leq \frac{1}{mq}$, hence $\bar x + \bar y \leq \bar z$.   If the vertex $(\bar x, \bar y) \in \intr(F)$, then  $\bar x + \bar y >u+v+\frac{1}{mq}$. We have $\bar x+\bar y \in K^-$, and
  $\Delta\phi(\bar x, \bar y) = (|I^+|-|K^+|)(s^+ -s^-) = \frac{1}{mq}(s_1- s_3)\geq 0$.
  If the vertex $(\bar x, \bar z - \bar x) \in \intr(F)$, then $\bar x - u > \bar z - u  - v - \frac{1}{mq}$.  We have  $\bar z - \bar x \in J^+$, 
  and  $\Delta\phi(\bar x, \bar z -\bar x) = (2|K^-| - |J^-|)(s^+ -s^-) = \frac{1}{mq}[(s^+ - s_3) + (s_2-  s_3)]\geq 0$.
  The vertex $(\bar z - \bar y, \bar y) \not\in \intr(F)$, since $\bar z - \bar y \geq u+\frac{1}{mq}$.
\end{proof}

Now we turn to Cases (b) and (b'); see again \autoref{fig:phi-subadd-F-b-bprime}. 
\begin{figure}[t]
  \centering
  \includegraphics[width=.45\linewidth]{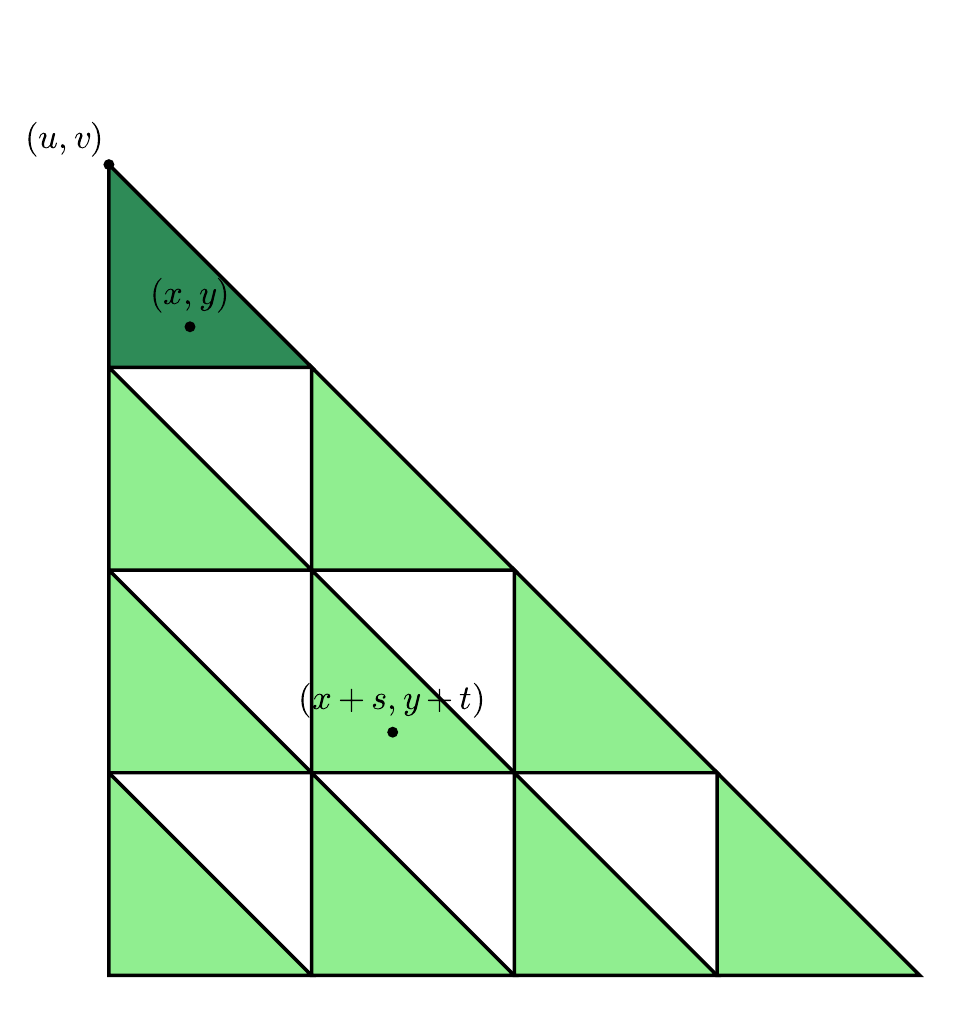}\hfill
  \includegraphics[width=.45\linewidth]{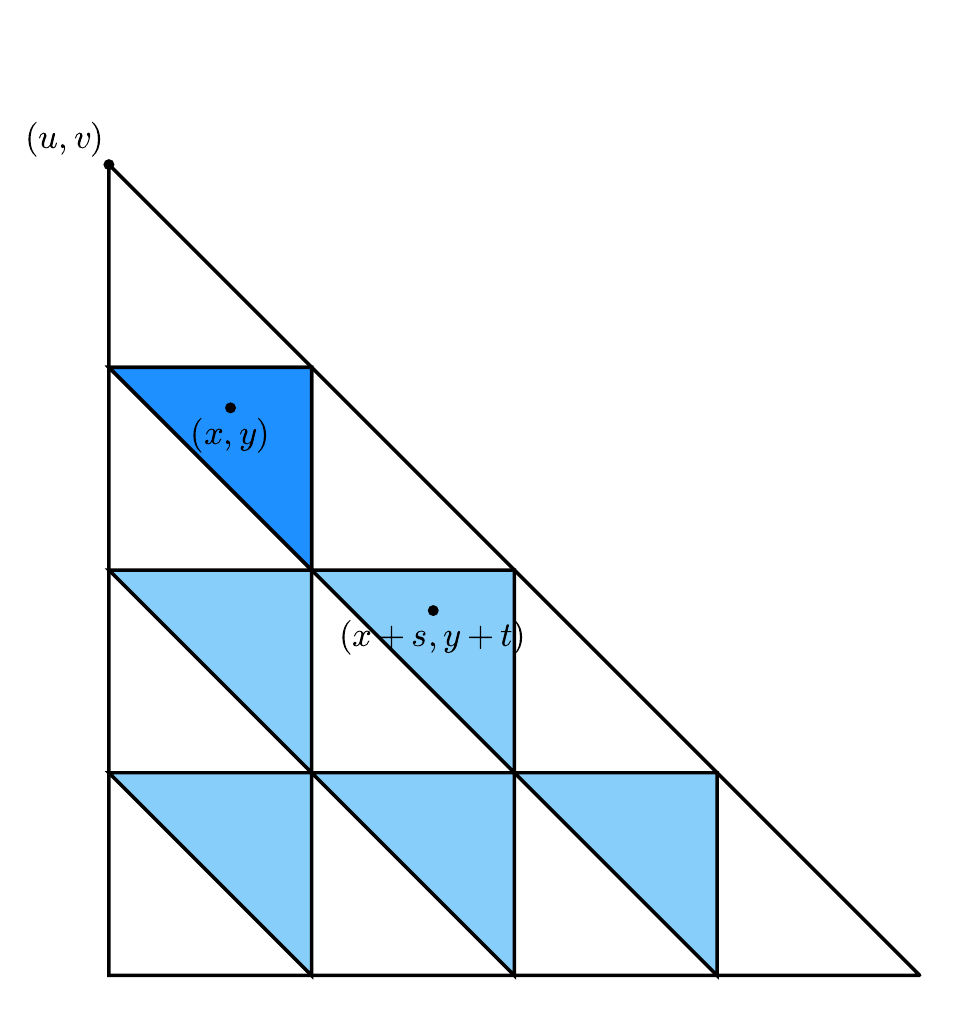}
  \caption{Cases (b) and (b') in the proof of \autoref{lemma:phi-subadditive}
    (Claims \ref{claim:phi-subadd-F-b} and \ref{claim:phi-subadd-F-b'}).
    The canonical triangle~$F$ is \emph{shaded dark}, its translations
    $F+(s,t)$  within
    the larger triangle $F'$ according
    to \autoref{obs:shifting-back-in-tangent-cone} are \emph{shaded lighter}.
  }
  \label{fig:phi-subadd-F-b-bprime}
\end{figure}

\begin{claim}
  \label{claim:phi-subadd-F-b}
  For the canonical lower triangle $F$ of Case (b), $\Delta\phi(x,y)\geq 0$ for $(x, y) \in F$.
\end{claim}

\begin{proof}
Again it suffices to prove the claim for $(x, y) \in \intr(F)$.  So we have 
\begin{equation}
\label{eq:case-b}
u < x < u +\tfrac{1}{mq}, \quad v-\tfrac{1}{mq} < y < v \quad  \text{ and  } \quad u+v-\tfrac{1}{mq}< x+y < u+v.
\end{equation}
From the additivity $\pi(u+v) = \pi(u) + \pi(v)$ and the subadditivity
$\pi(u+v) \leq \pi(u+\frac1{mq}) + \pi(v-\frac1{mq})$ 
we have that $\pi'(x) = mq \bigl( \pi(u + \frac1{mq}) - \pi(u) \bigr)
\geq mq \bigl( \pi(v) - \pi(v - \frac1{mq}) \bigr) = 
\pi'(y)$ and likewise $\pi'(y) \leq \pi'(x+y)$. Therefore, 
 \begin{align*}
d^+(x)\geq d^+(y) \quad \text{ and }\quad d^+(y)\leq d^+(x+y); \\
d^-(x)\leq d^-(y) \quad \text{ and }\quad d^-(y)\geq d^-(x+y).
 \end{align*}
We distinguish the following four subcases:
\begin{enumerate}
\item[(b\oldstylenums1)] $\phi(x)$ and $\phi(y)$ are both obtained according to Case \eqref{eq:ii};
\item[(b\oldstylenums2)] $\phi(x)$ and $\phi(y)$ are both obtained according to Case \eqref{eq:i};
\item[(b\oldstylenums3)] $\phi(x)$ and $\phi(y)$ are obtained according to Case \eqref{eq:i} and Case \eqref{eq:ii}, resp.;
\item[(b\oldstylenums4)] $\phi(x)$ and $\phi(y)$ are obtained according to Case \eqref{eq:ii} and Case \eqref{eq:i}, resp.
\end{enumerate}
We first show the subcase (b\oldstylenums1) cannot happen. Suppose that $\phi(x)$ and $\phi(y)$ are both obtained according to Case \eqref{eq:ii}. Then, $\pi(x) \geq \frac{1}{2}$ and $\pi(y) \geq \frac{1}{2}$. It follows from \autoref{prop:affine-same-sign} that $\pi(u) \geq \frac{1}{2}$ and $\pi(v) \geq \frac{1}{2}$. Since $\pi(u)+\pi(v)=\pi(u+v)\leq 1$, we obtain that $\pi(u)=\pi(v)=\frac{1}{2}$. Hence, $\pi(u+v)=1$ and $\{u+v\}=f$. Let $x' = u+v-x$. Then, $x' \in (v- \frac{1}{mq}, v)$ by equation \eqref{eq:case-b}, and $\pi(x')=1-\pi(x)$ by the symmetry condition of the minimal function $\pi$. If $\pi(x)> \frac{1}{2}$, then $\pi(x')<\frac{1}{2}$. We obtain from \autoref{prop:affine-same-sign} that $\pi(y)<\frac{1}{2}$ as well, a contradiction to $\pi(y) \geq \frac{1}{2}$. Therefore, $\pi(x) =\frac{1}{2}$. Similarly, we can prove that $\pi(y)=\frac{1}{2}$. Since $\phi(x)$ and $\phi(y)$ are both obtained according to Case \eqref{eq:ii}, $x,y \not\in \frac{1}{rq}\Z$ and $\frac{f}{2}, \frac{f+1}{2} \in \frac{1}{rq}\Z$, we have $\{x\}, \{y\} \in (\frac{f}{2}, \frac{f+1}{2})$. As $u,v \in \frac{1}{rq}\Z$, this implies that $\{u\} \in [\frac{f}{2}, \frac{f+1}{2}-\frac{1}{rq}]$ and $\{v\} \in [\frac{f}{2}+\frac{1}{rq}, \frac{f+1}{2}]$, and hence $\{u\}+\{v\} \in [f+\frac{1}{rq}, f+1-\frac{1}{rq}]$.  This is a contradiction to $\{u+v\}=f$. We conclude that Case (b\oldstylenums1) is impossible. Next, we will show that $\Delta\phi(x,y)\geq 0$ holds in the other three subcases.

In the subcase (b\oldstylenums2), the slope of $\phi$ on the intervals
$\big(u, u+d^+(x)\big)$, $\big(u+d^+(x), u+d^+(x)+d^-(x)\big)$,
$\big(v-d^-(y)-d^+(y), v-d^-(y)\big)$ and $\big(v-d^-(y), v\big)$ are $s^+$,
$s^-$, $s^+$ and  $s^-$, respectively. If $u+v-x-y > d^-(y) \geq d^-(x+y)$, then $x-u < \frac{1}{mq} - d^-(y) = d^+(y) \leq d^+(x)$ and  $v-y > d^-(y)$. It follows that
\begin{align*}
\phi(x) &= \pi(u)+s^+\cdot (x-u), \\
\phi(y) &=\pi(v)-s^-\cdot d^-(y) - s^+\cdot \big(v-y-d^-(y)\big).
\end{align*}
The length of the subinterval of $[x+y, u+v]$ where $\phi$ has slope~$s^-$ 
  is bounded above by $d^-(x+y)$.  Hence, we have
\begin{align*}
\phi(x+y) \leq  \pi(u+v)-s^-\cdot d^-(x+y) - s^+\cdot
            \big(u+v-x-y-d^-(x+y)\big),
\end{align*}
which we estimate further to obtain
\begin{align*}
\phi(x+y) \leq \pi(u+v)-s^-\cdot d^-(y) - s^+\cdot \big(u+v-x-y-d^-(y)\big),
\end{align*}
and hence $\Delta\phi(x,y) \geq \Delta\pi(u,v)=0$.
Otherwise, $u+v-x-y \leq d^-(y)$. Then, 
\begin{align*}
\phi(x)+\phi(y) &\geq \pi(u)+\pi(v)+s^-\cdot (x+y-u-v), \\
\phi(x+y) &\leq \pi(u+v)+s^-\cdot (x+y-u-v).
\end{align*}
Hence, we also have that  $\Delta\phi(x,y)\geq \Delta\pi(u,v)= 0$.

In the subcase (b\oldstylenums3), we have $\pi(y) \geq \frac{1}{2}$, and hence $\pi(v) \geq \frac{1}{2}$ by \autoref{prop:affine-same-sign}. If $\pi(u)=0$, then $\{u\}=0$ and $\phi(x)= s^+ \cdot \{x\}$. Since $\phi(x+y) - \phi(y) \leq s^+ \cdot \{x\}$, we obtain that $\Delta\phi(x,y)\geq 0$. Now we assume that $\pi(u)>0$. Since $\Delta\pi(u,v)=0$, we have $\pi(u+v) = \pi(v)+\pi(u) >\pi(v) \geq \frac{1}{2}$. By \autoref{prop:affine-same-sign}, $\pi(x+y) >\frac{1}{2}$ as well. Hence, $\phi(x+y)$ is obtained according to Case \eqref{eq:ii}. We note that the function
$\phi$ has slope $s^+$ on the intervals $\big(v-d^+(y), v\big)$, $\big(u, u+d^+(x)\big)$ and $\big(u+v-d^+(x+y), u+v\big)$, and slope $s^-$ on the intervals $\big(v-\frac{1}{mq}, v-d^+(y)\big)$, $\big(u+d^+(x), u+\frac{1}{mq}\big)$ and $\big(u+v-\frac{1}{mq}, u+v-d^+(x+y)\big)$, with $d^+(y) \leq d^+(x+y)$ and $d^+(y) \leq d^+(x)$. 
Thus, 
\begin{align*}
\phi(x) - \phi(u) &\geq \phi(v)-\phi(u+v-x) \text{ and } \\
\phi(x+y)-\phi(u+v) &\leq \phi(x+y-u)-\phi(v).
\end{align*}
It follows from equation \eqref{eq:case-b} that $y \leq u+v-x \leq v$ and  $y \leq x+y-u \leq v$. Since $\phi(y)$ is obtained by equation \eqref{eq:ii}, we know that the function $\phi$ is convex on the interval $[y, v]$.  Notice that $y+v = (u+v-x)+(x+y-u)$. 
We obtain that 
\[\phi(y)+\phi(v) \geq \phi(u+v-x)+\phi(x+y-u).\] 
We also have that $\Delta\phi(u,v)=\Delta\pi(u,v)=0$.
Therefore,
\begin{align*}
\Delta\phi(x,y) &=\phi(x)+\phi(y)-\phi(x+y) \\
& \geq [\phi(u)+\phi(v)-\phi(u+v-x)]+\phi(y)-[\phi(u+v)+\phi(x+y-u)-\phi(v)] \\
&= [\phi(u)+\phi(v)-\phi(u+v)] + [\phi(y)+\phi(v)] - [\phi(u+v-x)+\phi(x+y-u)] \\
&=\Delta\pi(u,v) + [\phi(y)+\phi(v)] - [\phi(u+v-x)+\phi(x+y-u)] \\
&\geq 0.
\end{align*}

In the subcase (b\oldstylenums4), we have $\pi(x) \geq \frac{1}{2}$, and hence $\pi(u) \geq \frac{1}{2}$ by \autoref{prop:affine-same-sign}. If $\pi(v)=0$, then $\{v\}=0$ and  $\phi(y)= s^- \cdot (y-v)=s^- \cdot \big(\{y\}-1\big)$. Since $\phi(x+y) - \phi(x) \leq s^-\cdot \big(\{y\}-1\big)$, we obtain that $\Delta\phi(x,y)\geq 0$. Now we assume that $\pi(v)>0$. It follows from $\Delta\pi(u,v)=0$ that $\pi(u+v) = \pi(v)+\pi(u) >\pi(u) \geq \frac{1}{2}$. By \autoref{prop:affine-same-sign}, $\pi(x+y) >\frac{1}{2}$. Hence, $\phi(x+y)$ is obtained according to Case \eqref{eq:ii}. We note that the function 
$\phi$ has slope $s^-$ on the intervals $\big(v-d^-(y), v\big)$, $\big(u, u+d^-(x)\big)$ and $\big(u+v-d^+(x+y), u+v\big)$, and slope $s^+$ on the intervals $\big(v-\frac{1}{mq}, v-d^-(y)\big)$, $\big(u+d^-(x), u+\frac{1}{mq}\big)$ and $\big(u+v-\frac{1}{mq}, u+v-d^+(x+y)\big)$, with $d^-(y) \geq d^-(x+y)$ and $d^-(y) \geq d^-(x)$. 
Thus,
\begin{align*}
\phi(x) - \phi(u) &\geq \phi(v)-\phi(u+v-x)  \text{ and } \\
\phi(u+v)-\phi(x+y) &\geq \phi(v-\tfrac{1}{mq}+u+v-x-y)-\phi(v-\tfrac{1}{mq}).
\end{align*}
Since $\Delta\phi(u,v)=\Delta\pi(u,v)=0$, we obtain that
\begin{align*}
\Delta\phi(x,y) =\:& \phi(x)+\phi(y)-\phi(x+y) \\
\geq\:& \phi(u)+\phi(v)-\phi(u+v-x)+\phi(y) \\
 & -\phi(u+v)+\phi(v-\tfrac{1}{mq}+u+v-x-y)-\phi(v-\tfrac{1}{mq}) \\
 =\:& \phi(y) +\phi(v-\tfrac{1}{mq}+u+v-x-y)-\phi(v-\tfrac{1}{mq})-\phi(u+v-x).
\end{align*}
Notice that $y + (v-\tfrac{1}{mq}+u+v-x-y) = (v-\tfrac{1}{mq})+(u+v-x)$ and $v-\tfrac{1}{mq} \leq y \leq u+v-x \leq v$ by equation \eqref{eq:case-b}. Since $\phi(y)$ is obtained by equation \eqref{eq:i}, we know that the function $\phi$ is concave on the interval $[v-\tfrac{1}{mq}, v]$. Therefore, $\Delta\phi(x,y)\geq 0$.

We proved that $\Delta\phi(x,y)\geq 0$ in Case (b).
\end{proof}

It remains to prove the following for Case (b').

\begin{claim}\label{claim:phi-subadd-F-b'}
  Let $F$ be the canonical upper triangle of Case (b').  Then
  $\Delta\phi \geq 0$ on all vertices of the complex $\Delta\P^\phi$ that lie
  in $\intr(F)$.
\end{claim}
\begin{proof}
  We have $F = F(I,J,K)$ with  $I = [u, u+\frac1{mq}]$,
  $J=[v-\frac2{mq},v-\frac1{mq}]$, and $K=[u+v-\frac1{mq}, u+v]$.
  Let $s_1, s_2$ and $s_3$ denote the slopes of $\pi$ on the intervals $I, J$
  and $K$, respectively. 
  As in the proof above, we have that $s_2 \leq s_1$ and $s_2 \leq s_3$.
  Let $\bx$, $\by$, and $\bz$ denote the unique breakpoints of $\phi$ in $\intr(I)$,
  $\intr(J)$, and $\intr(K)$, respectively. 
  Then the only vertices of
  $\Delta\P^\phi$ that can lie in $\intr(F)$ are 
  $(\bx, \by)$,       
  $(\bx, \bz-\bx)$, and 
  $(\bz-\by, \by)$.     
  The breakpoint $\bx$ splits $I$ into an interval $I^+$, on which
  $\phi$ has slope $s^+$, and $I^-$, on which $\phi$ has slope $s^-$. Likewise, we
  obtain intervals $J^+, J^-, K^+, K^-$, whose lengths are given by
  \eqref{eq:def-d}.  
  
  We distinguish the same subcases as in Case (b). 

  Subcase (b'\oldstylenums1) cannot happen, using the same proof.

  In subcase (b'\oldstylenums2), $\phi(x)$ and $\phi(y)$ are both obtained
  by \eqref{eq:i}.  If $\phi(x+y)$ is obtained by \eqref{eq:ii}, then
  trivially $\Delta\phi(x,y) \geq \Delta\pi(x,y) \geq 0$ for all $(x,y) \in F$. 
  So we will assume that $\phi(x+y)$ is obtained by \eqref{eq:i}. 
  
  The vertex $(\bz-\by, \by)$ does not lie in~$\intr(F)$ because $\bz-\by \geq u +
  \frac1{mq}$. 
  
  If the vertex $(\bx, \bz-\bx)$ lies in $\intr(F)$, then $\bz - \bx < v -
  \frac1{mq}$. Also, $\bz-\bx \geq \by$, and thus 
  $\bz-\bx \in J^-$.  
  Then a calculation shows $\Delta\phi(\bx, \bz-\bx) \geq 0$. 

  If the vertex $(\bx, \by)$ lies in $\intr(F)$, then $\bx + \by > u + v -
  \frac1{mq}$.  Also $\bx + \by < \bz$, and thus $\bx + \by \in J^+$. 
  Then a calculation shows $\Delta\phi(\bx, \by) \geq 0$. 

  In subcase (b'\oldstylenums3), $\phi(x)$ is obtained by  \eqref{eq:i},
  and $\phi(y)$ is obtained by  \eqref{eq:ii}. 
  As in (b\oldstylenums3), if $\pi(u) = 0$, $\Delta\phi(x,y)\geq0$ follows
  easily for all $(x,y)\in F$, and if $\pi(u) > 0$, then $\phi(x+y)$ is obtained using
  \eqref{eq:ii}.
  Then the vertices $(\bx, \by)$ and $(\bz-\by, \by)$ lie in~$F$. 
  There are two combinatorial types. 

  In Type I, we have $\bz - \by > \bx$.
  Then $\Delta\phi(\bx,\by) = \frac1{mq} ((s_1 - s_2) + (s_3 - s_2)) \geq 0$ and 
  $\Delta\phi(\bz-\by, \by) = \frac1{mq} (s^+ - s_2) \geq 0$. 
  Finally, if the vertex $(\bx, \bz-\bx)$ lies in $\intr(F)$, then $\bz - \bx < v -
  \frac1{mq}$ and $\bz - \bx \in J^+$, and then $\Delta\phi(\bx,\bz-\bx) =
  \frac1{mq} (s^+ - s_2) \geq 0$. 

  In Type II, we have $\bz - \by \leq \bx$.
  Then all three vertices lie in~$F$.  We have $\bx + \by \in K^+$ and
  $\Delta\phi(\bx, \by) = \frac1{mq} (s^+ - s_2) \geq 0$.
  We have $\bz-\bx \in J^-$ and 
  $\Delta\phi(\bx, \bz-\bx) = \frac1{mq} ((s_1 - s_2) + (s_3 - s_2)) \geq 0$.
  We have $\bz-\by \in I^+$ and $\Delta\phi(\bz-\by, \by) = \frac1{mq} (s^+ - s_2) \geq 0$.
  
  In subcase (b'\oldstylenums4), $\phi(x)$ is obtained by \eqref{eq:ii}
  and $\phi(y)$ is obtained by  \eqref{eq:i}. 
  As in (b\oldstylenums4), if $\pi(v) = 0$, $\Delta\phi(x,y)\geq0$ follows
  easily, and if $\pi(v) > 0$, then $\phi(x+y)$ is obtained using
  \eqref{eq:ii}.

  The vertex $(\bx, \by)$ has $\bx + \by \leq u + v - \frac1{mq}$, and so $(\bx,\by)
  \notin \intr(F)$. 

  If the vertex $(\bz-\by, \by)$ lies in $\intr(F)$, then $\bz - \by \in I^+$ and 
  $\Delta\phi(\bz-\by, \by) =  \frac1{mq} ((s_1 - s_2) + (s^+ - s_2)) \geq 0$. 

  If the vertex $(\bx, \bz-\bx)$ lies in $\intr(F)$, then $\bz - \bx \in J^-$ and 
  $\Delta\phi(\bx, \bz-\bx) = \frac1{mq} (s_3 - s_2) \geq 0$. 
\end{proof}

From Claims \ref{claim:phi-subadd-F-strictly-subadd} and
\ref{claim:phi-subadd-F-dim-2}--\ref{claim:phi-subadd-F-b'}, we conclude that
$\phi$ is a subadditive function.  This completes the proof of
\autoref{lemma:phi-subadditive}.
\end{proof}

\begin{theorem}
\label{thm:phi-extreme}
The function $\phi$ defined above is an extreme function for $R_f(\R/\Z)$.
\end{theorem}
\begin{proof}
It follows from \autoref{lemma:phi-prop}, \autoref{lemma:phi-non-negative} and \autoref{lemma:phi-subadditive} that $\phi$ is a minimal valid function that has two slopes. Therefore, by Gomory--Johnson's Two Slope Theorem \cite{infinite}, $\phi$ is an extreme function.
\end{proof}

\begin{proof}[Proof of \autoref{thm:approx-theorem}]
  The theorem follows from \autoref{thm:phi-approximates},
  \autoref{lemma:phi-prop}, and \autoref{thm:phi-extreme}.
\end{proof}

\section{Examples}
\label{s:examples}

Figures~\ref{fig:phi-and-pi-sym} and \ref{fig:phi-and-pi-sym-2} show the graphs of the functions $\phi =
\sage{injective\_2\_slope\_fill\_in}(\pi)$ constructed in the present paper (left)
and $\pi_{\text{sym}} = \sage{symmetric\_2\_slope\_fill\_in}(\pi)$ constructed in \cite{bhm:dense-2-slope} (right). 

Both approximation procedures are implemented in the latest version of our software
package \textsf{cutgeneratingfunctionology}
\cite{cutgeneratingfunctionology:lastest-release}, a version of which is
described in \cite{hong-koeppe-zhou:software-paper}.  The reader
is invited to try more examples.

\section{Proofs of corollaries}
\label{s:proofs-corollaries}

\begin{proof}[Proof of \autoref{th:infinite-master-for-facets}]
  Let $\pi|_P$ be a facet of the finite group problem.  By Gomory's master
  theorem for facets, \autoref{thm:gomory-master}, there exists an extension
  $\pi|_G$ of $\pi|_P$ that is a facet of the finite master group problem $R_f(G)$.
  Because $\pi|_G$ is an extreme point of the polyhedron described by the
  rational inequality system \eqref{eq:minimal}, it takes rational values on
  $G$.  By our injective approximation theorem, \autoref{thm:approx-theorem},
  applied to $\pi = \sagefunc{interpolate_to_infinite_group}(\pi|_G)$ and an
  arbitrary $\epsilon>0$, there exists an extension $\phi$ of $\pi|_G$ that is
  extreme for $R_f(\R/\Z)$.  Then $\phi|_P = \pi|_P$.
\end{proof}
\smallbreak

\begin{proof}[Proof of \autoref{cor:non-extreme-limit}]
  The corollary follows from \autoref{thm:approx-theorem} by choosing a
  sequence $\{ \epsilon_i \}$ converging to $0$ and a sequence of integers
  $m_i \geq m_\pi(\epsilon_i)$ such that $\lcm(1, \dots, i) \mid m_i $.  Let
  $\phi_i$ be the function from \autoref{thm:approx-theorem} for parameters
  $m_i$ and $\epsilon_i$. Then for every rational number $x = a/b$, for all
  $i \geq b$ we have $x \in \frac1{m_iq}\Z$ and hence $\phi_i(a/b) = \pi(a/b)$
  by \autoref{thm:approx-theorem}, property~(iv).
\end{proof}

\section*{Acknowledgments}

The authors wish to thank Joseph Paat for discussions on related topics
and the anonymous referees for their helpful comments.

\appendix

\section{Example: Restrictions of extreme functions to subgroups are not
  necessarily extreme}
\label{s:restriction-extreme-not-extreme}

As noted in the introduction,  restrictions of continuous piecewise linear extreme
functions with breakpoints on $\frac{1}{mq}\Z$ to the finite group problem on
$\frac{1}{mq}\Z$ are necessarily extreme, but in general the coarser
restrictions to the finite group problems on $\frac{1}{q}\Z$ are \emph{not} extreme. We
give a counterexample for this in \autoref{fig:counterexample-to-survey-thm-82}.

{\small
\begin{verbatim}
sage: h = gj_2_slope(1/2,1/3)
sage: h24 = restrict_to_finite_group(h, order=24)
sage: h8 = restrict_to_finite_group(h, order=8)
sage: extremality_test(h24)
True
sage: extremality_test(h8)
False
\end{verbatim}
}

\begin{figure}[h]
\includegraphics[width=.48\textwidth]{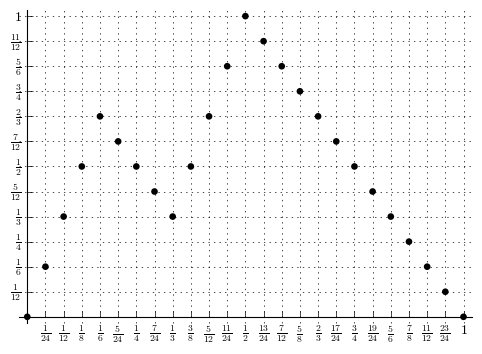}
\includegraphics[width=.48\textwidth]{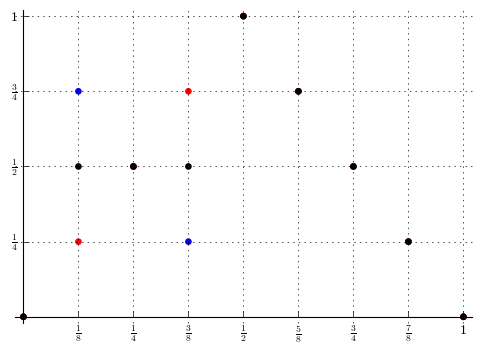}
\caption{The function $\pi$ is from a family of 2-slope extreme functions discovered by Gomory--Johnson. It is continuous piecewise linear with breakpoints $0, \frac{1}{6}, \frac{1}{3}, \frac{1}{2}$ and $1$, which all belong to $\frac{1}{24}\Z$. The function restricted to $\frac{1}{24}\Z$ (left) is an extreme function for the finite group problem on $\frac{1}{24}\Z$. However, the coarser
restriction to the finite group problems on $\frac{1}{8}\Z$ (right) is not extreme, as it is the convex combination of the  two minimal valid functions shown in \emph{blue} and \emph{red}.
}
\label{fig:counterexample-to-survey-thm-82}
\end{figure}

\clearpage
\begin{figure}[tp]
\includegraphics[width=.38\textwidth]{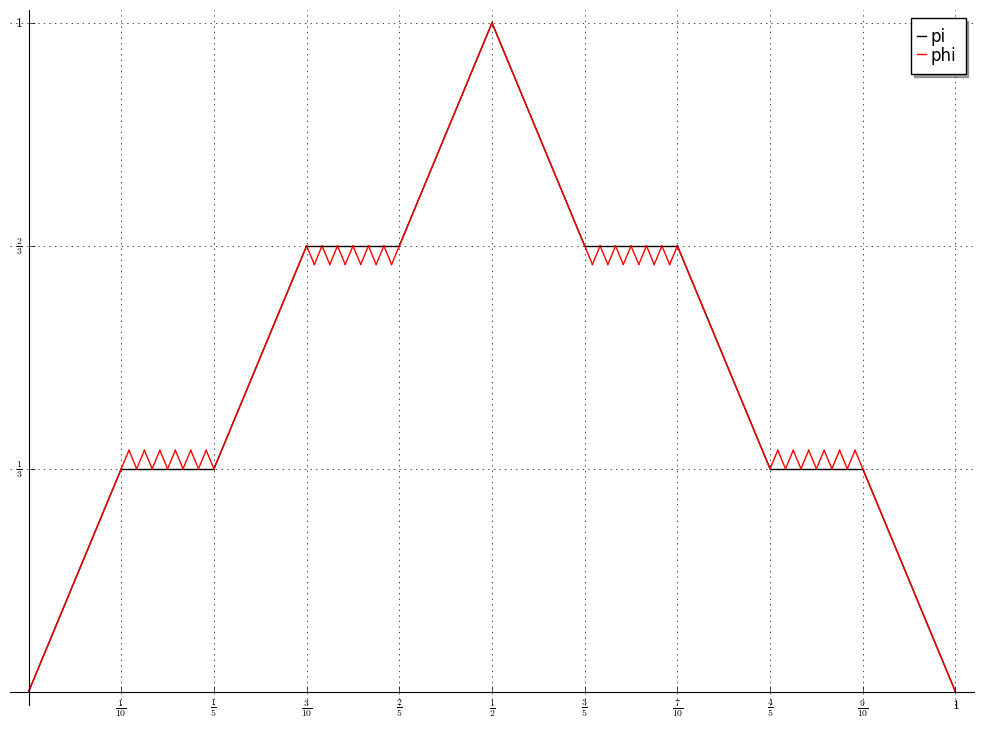}
\includegraphics[width=.38\textwidth]{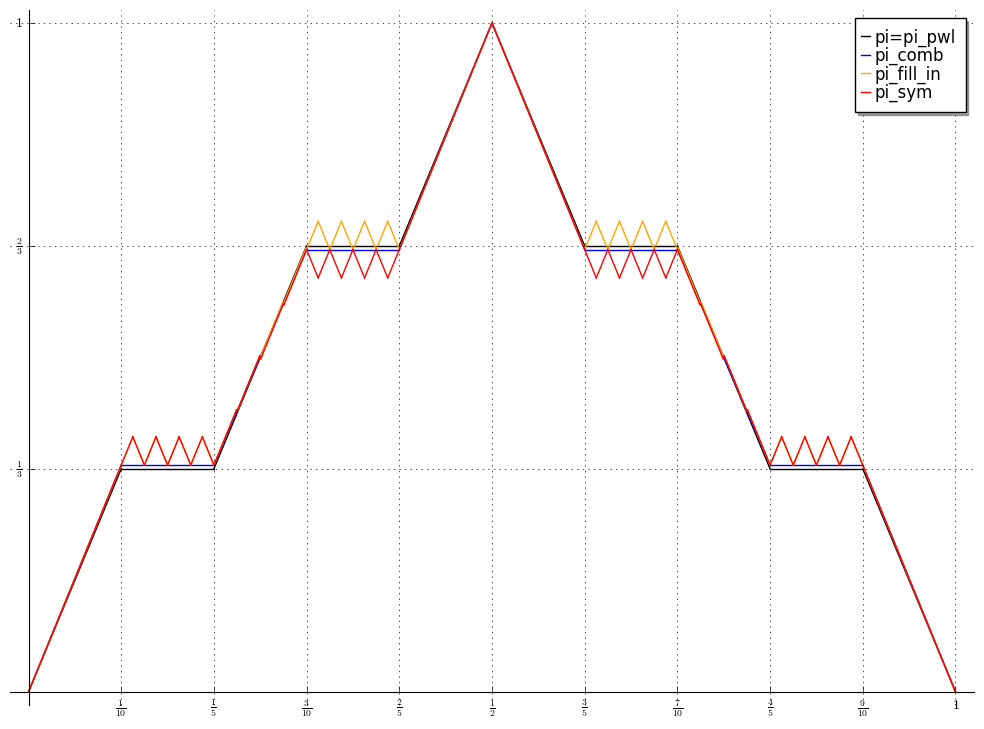}
\includegraphics[width=.38\textwidth]{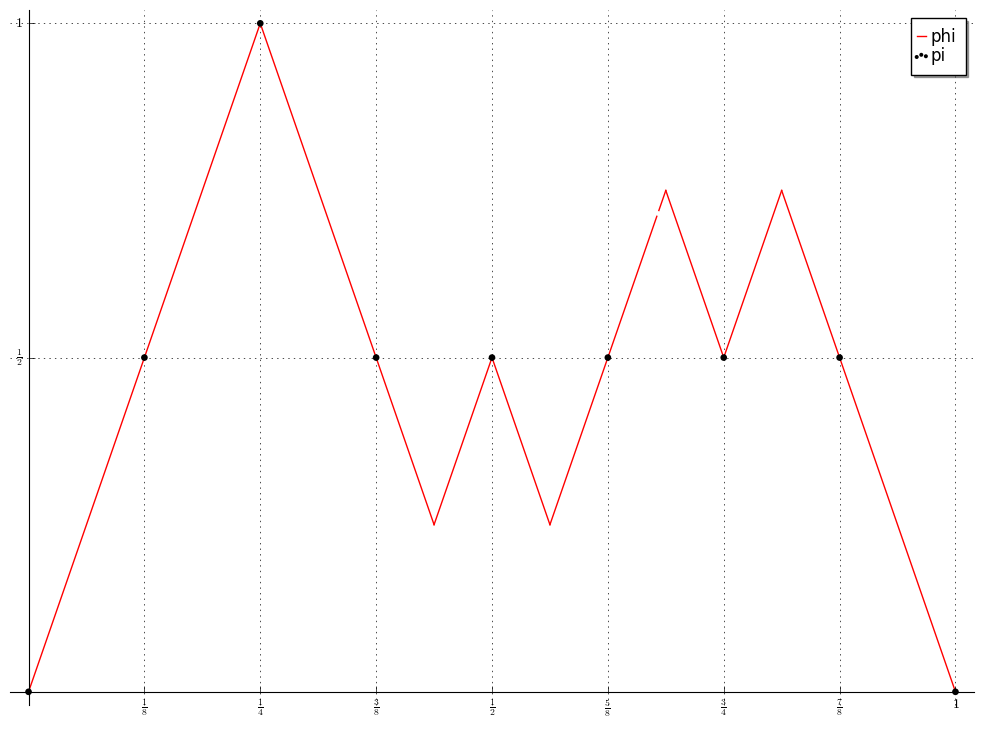}
\includegraphics[width=.38\textwidth]{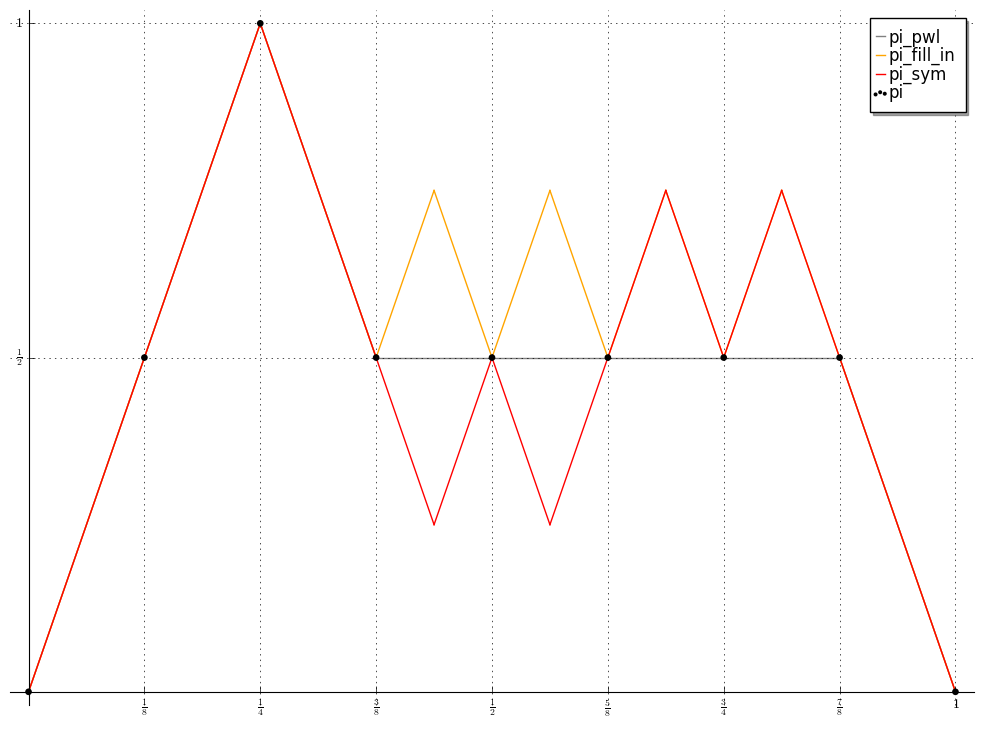}
\includegraphics[width=.38\textwidth]{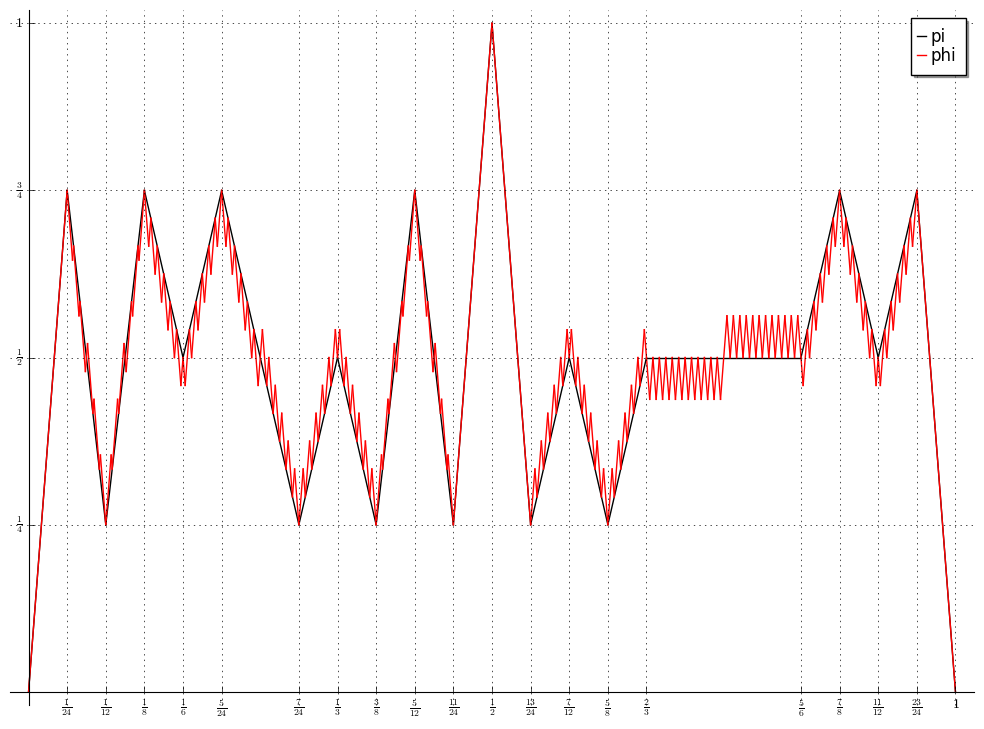}
\includegraphics[width=.38\textwidth]{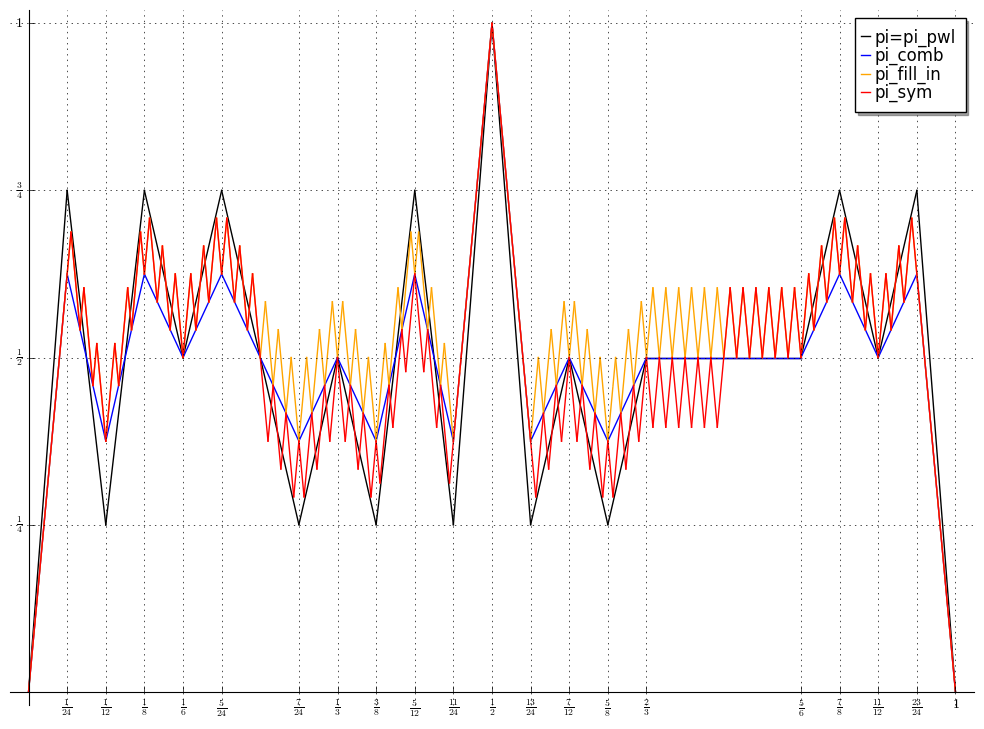}
\includegraphics[width=.38\textwidth]{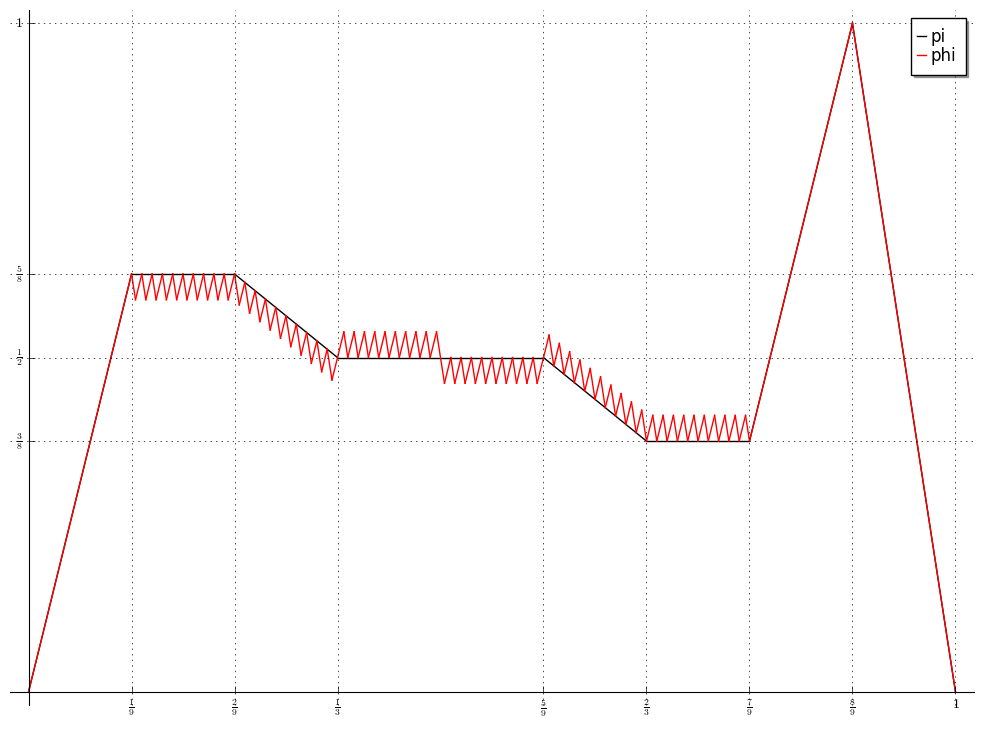}
\includegraphics[width=.38\textwidth]{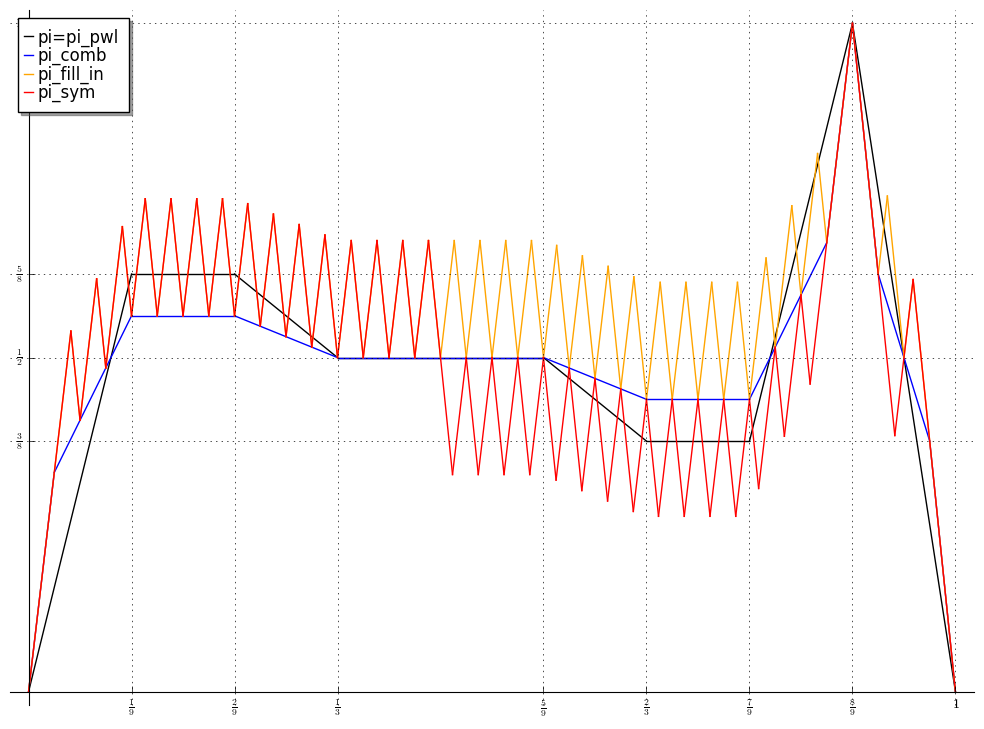}
\includegraphics[width=.38\textwidth]{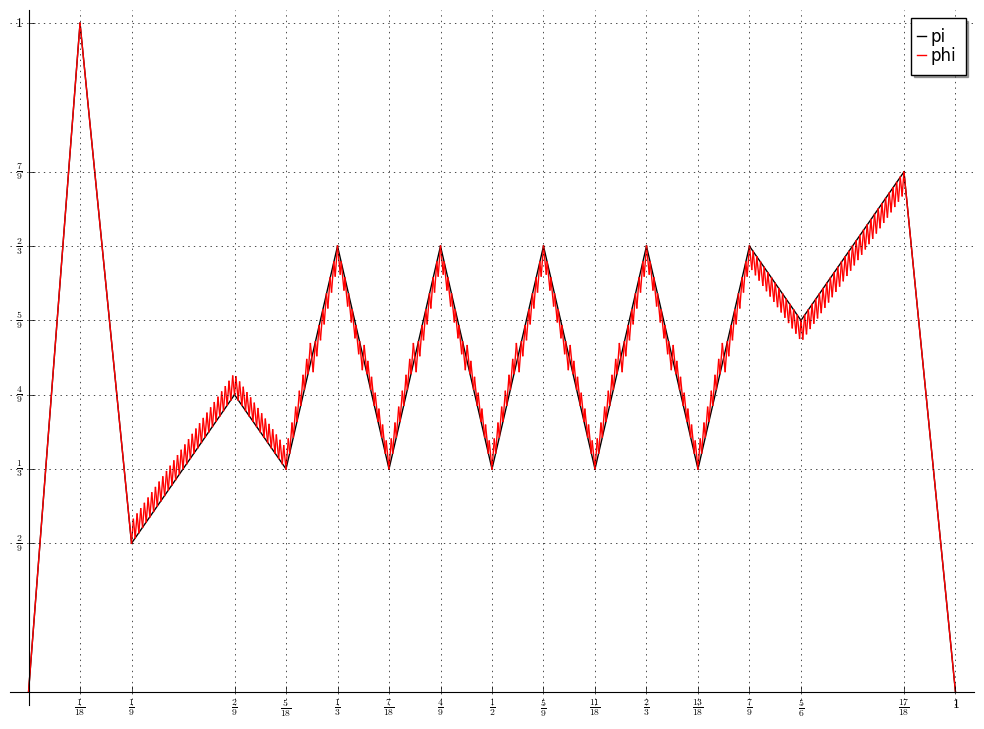}
\includegraphics[width=.38\textwidth]{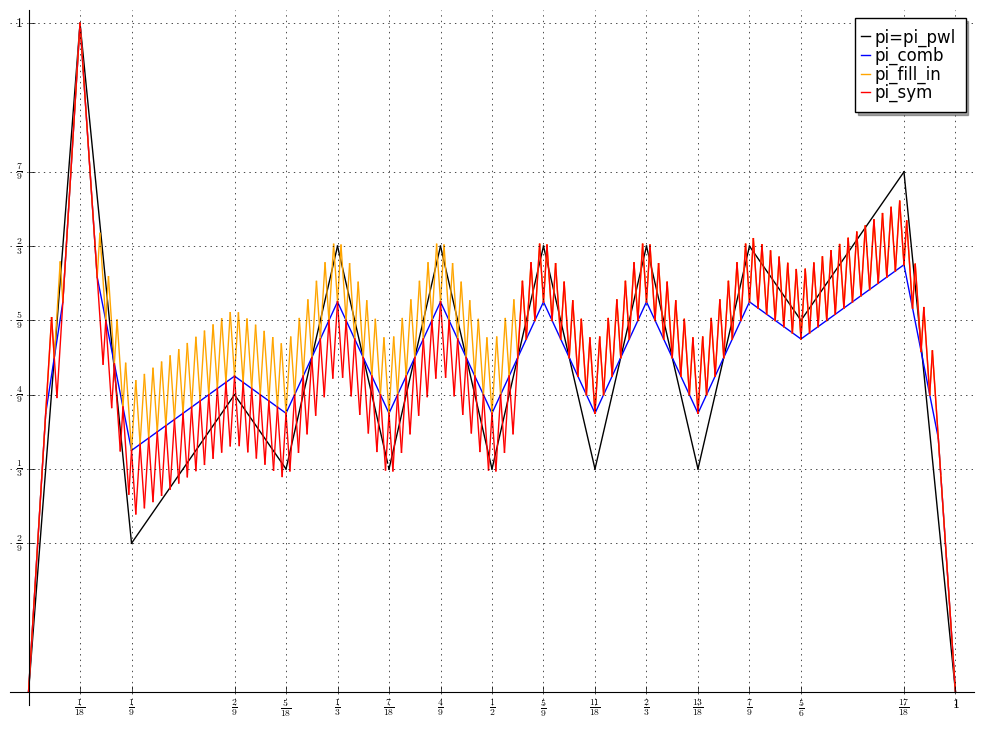}
\caption{Graphs of the approximating extreme functions (left)
  $\phi=\sage{injective\_2\_slope\_fill\_in}(\pi)$ from the present paper
  and (right)
  $\pi_{\text{sym}}=\sage{symmetric\underscore 2\underscore slope\underscore
    fill\underscore in}(\pi)$ from Basu--Hildebrand--Molinaro
  for several examples~$\pi$.  As noted in the introduction, in contrast to
  our approximation $\phi$, Basu et al.'s approximation
  $\pi_{\text{sym}}$ does not preserve the function values on the group
  $\frac{1}{q}\Z/\Z$.}
\label{fig:phi-and-pi-sym-2}
\end{figure}

\clearpage

\providecommand\ISBN{ISBN }
\bibliographystyle{../amsabbrvurl}
\bibliography{../bib/MLFCB_bib}
\end{document}